\documentclass{amsart}
\usepackage{amsmath}
\usepackage{amssymb}
\usepackage{amsthm}
\usepackage{enumerate}
\usepackage[pdftex]{graphicx}
\usepackage{caption}
\theoremstyle{definition}
\newtheorem{definition}{Definition}[section]
\theoremstyle{plain}
\newtheorem{lemma}[definition]{Lemma}
\newtheorem{theorem}[definition]{Theorem}

\newtheorem{proposition}[definition]{Proposition}
\newtheorem{corollary}[definition]{Corollary}
\theoremstyle{remark}
\newtheorem{remark}[definition]{Remark}

\newtheorem{example}[definition]{Example}

\makeatletter
\@namedef{subjclassname@2020}{
  \textup{2020} Mathematics Subject Classification}
\makeatother

\newcommand{\myint}{\operatorname{int}}
\newcommand{\mycl}{\operatorname{cl}}
\newcommand{\convl}{\operatorname{Conv}_{\mathrm{left}}}
\newcommand{\convr}{\operatorname{Conv}_{\mathrm{right}}}
\newcommand{\awk}{\preceq_{\text{awk}}}
\newcommand{\sgn}{\operatorname{sgn}}
\newcommand{\myindr}{\mathcal R_{\text{ind}}}
\newcommand{\myds}{\operatorname{D}_\Sigma}

\begin{document}
\title[Definable compactness]{Definable compactness in definably complete locally o-minimal structures}
\author[M. Fujita]{Masato Fujita}
\address{Department of Liberal Arts,
Japan Coast Guard Academy,
5-1 Wakaba-cho, Kure, Hiroshima 737-8512, Japan}
\email{fujita.masato.p34@kyoto-u.jp}

\begin{abstract}
We demonstrate that And\'ujar Guerrero, Thomas and Walsberg's results on definable compactness in o-minimal structures still hold true in definably complete locally o-minimal structures.
As an application, we show that a definably simple definable topological group which is regular, Hausdorff and definably compact as a definable topological space is either discrete or definably connected.  
We also study the definable quotient of definable continuous actions by definably compact definable topological groups.
\end{abstract}



\maketitle
\section{Introduction}\label{sec:intro}
Local o-minimal structure localizes the definition of o-minimal structures.
The main targets of this paper are locally o-minimal structures.
Readers who have interest in o-minimal structures should consult van den Dries's book \cite{vdD}.
Locally o-minimal structures  were first studied by Toffalori and Vozoris \cite{TV}.
The works conducted by Fornasiero \cite{F} and Kawakami et al. \cite{KTTT} followed it.
The author and his collaborators also have studied definable complete locally o-minimal structures and its relatives \cite{Fuji, Fuji3, Fuji4, Fuji6, Fuji7, FKK}.
In \cite{FKK}, we studied definably compact definable sets, which are main subjects of this paper, when definable bounded multiplication is available.

The studies of o-minimal structures are the source of conjectures for locally o-minimal structures.
A typical conjecture is whether properties which hold true for o-minimal structures are remained to hold true even in locally o-minimal structures.
We consider definable compactness in this paper.
It is proposed in \cite{PS} in the o-minimal setting.
And\'ujar  Guerrero, Thomas and Walsberg presented a new characterization of definably compact  definable sets in \cite{GTW} when the structure is o-minimal.
The question is whether the assertions in \cite{GTW} remain to hold true in definably complete locally o-minimal structures.
We give affirmative answers after a slight modification of the assertions.
It is the first contribution of this paper.

The second contribution of this paper is an application of the above result.
We show that a definably simple definable topological group which is definably compact is either discrete or definably connected.  
This assertion is a definable version of the well-known result that a simple topological group is either discrete or connected.
We do not need the assumption that the topological group is compact in the purely topological setting.
This classical fact immediately follows from the fact that a connected component of a topological group is a normal subgroup.
In the o-minimal setting, we can use the notion of definably connected components instead of that of connected components.
The classical proof works also in this case.
However, a definably connected component of a definable set is not still well-defined in a definably complete locally o-minimal structure.
For instance, we do not know whether a connected component of a definable set is definably connected in a locally o-minimal expansion of the set of reals $\mathbb R$.
We have to prove the above assertion without using the notion of definably connected components.
We also demonstrate that a definable topological group $G$ is definably compact if both its definable normal subgroup $H$ and the definable quotient $G/H$ are definably compact.
Finally, we show that the quotient $X/G$ of a definable metric space by a definably compact definable group $G$ is again a definable metric space when $G$ acts on $X$ definably and continuously.

This paper is organized as follows:
We need the fact that a definable continuous function on a definable, closed and bounded set is uniformly continuous.
It is proved in Section \ref{sec:preliminary}.
In Section \ref{sec:directed}, we first demonstrate that \cite[Theorem 8]{GTW} holds true even if the given structure is definably complete locally o-minimal expansion of an ordered group.
We recall the definition of definable topologies and give equivalence conditions for definable sets to be definably compact in Section \ref{sec:compact}.
Finally, Section \ref{sec:group} is devoted to the study of definable topological groups which are definably compact.

We introduce the terms and notations used in this paper. The term ‘definable’ means ‘definable in the given structure with parameters’ in this paper.
For an ordered group $(M,0,<)$, we set $M^{>0}:=\{x \in M\;|\; x > 0\}$.
Unless specified, we assume that $M$ equips the order topology induced from the linear order $<$ and the topology on $M^n$ is the product topology of the order topology on $M$.
For a topological space $(X,\tau)$ with the topology $\tau$, $\myint_{\tau}(A)$, $\mycl_{\tau}(A)$ and $\partial_{\tau}A$ denotes the interior, the closure and the frontier of a subset $A$ of $X$, respectively.
We drop the subscript $\tau$ when the topology $\tau$ is obvious from the context.

\section{Preliminary}\label{sec:preliminary}
We first recall the definitions of local o-minimality and definable completeness.
\begin{definition}[\cite{M,TV}]
	An expansion of a dense linear order without endpoints $\mathcal M=(M,<,\ldots)$ is \textit{locally o-minimal} if, for every definable subset $X$ of $M$ and for every point $a\in M$, there exists an open interval $I$ containing the point $a$ such that $X \cap I$ is  a finite union of points and open intervals.
	We can immediately show that $X \cap I$ is the union of at most one point and at most two open intervals if we choose $I$ appropriately.
	The expansion $\mathcal M$ is \textit{definably complete} if any definable subset $X$ of $M$ has the supremum and  infimum in $M \cup \{\pm \infty\}$.
\end{definition}

\begin{definition}
	Consider a structure.
	A \textit{definable equivalence relation} $E$ on a definable set $X$ is a definable subset of $X \times X$ such that the binary relation $\sim_E$ defined by $x \sim_E y \Leftrightarrow (x,y) \in E$ is an equivalence relation defined on $X$.
\end{definition}

The following definable choice lemma is often used in this paper.
\begin{lemma}[Definable choice lemma]\label{lem:definable_choice}
	Consider a definably complete locally o-minimal expansion of an ordered group $\mathcal M=(M,<,0,+\ldots)$.
Let $\pi:M^{m+n} \rightarrow M^m$ be a coordinate projection.
Let $X$ and $Y$ be definable subsets of $M^m$ and $M^{m+n}$, respectively,  satisfying the equality $\pi(Y)=X$.
There exists a definable map $\varphi:X \rightarrow Y$ such that $\pi(\varphi(x))=x$ for all $x \in X$.
Furthermore, if $E$ is a definable equivalence relation defined on a definable set $X$, there exists a definable subset $S$ of $X$ such that $S$ intersects at exactly one point with each equivalence class of $E$. 
\end{lemma}
\begin{proof}
	See \cite[Lemma 2.8]{Fuji7}.
\end{proof}

We demonstrate several assertions for definable complete expansions of an ordered group.
\begin{definition}	
	Consider an expansion of a dense linear order without endpoints $\mathcal M=(M;<,\ldots)$.
	Let $X$ and $T$ be definable sets.
	The parameterized family $\{S_t\;|\;t \in T\}$ of definable subsets of $X$ is called \textit{definable} if the union $\bigcup_{t \in T} \{t\} \times S_t$ is definable.

	For a set $X$, a family $\mathcal F$ of subsets of $X$ is called a \textit{filtered collection} if, for any $B_1, B_2 \in \mathcal F$, there exists $B_3 \in \mathcal F$ with $B_3 \subseteq B_1 \cap B_2$. 
	A parameterized family $\{S_t\;|\;t \in T\}$ of definable subsets of a definable set $X$ is a \textit{definable filtered collection} if it is simultaneously definable and a filtered collection.
\end{definition}

\begin{lemma}\label{lem:prepre}
	Consider a definably complete expansion of an ordered group $\mathcal M=(M,<,+,0,\ldots)$.
	Let $C$ be a definable, closed and bounded subset of $M^m$.
	Let $\varphi, \psi: C \rightarrow M^{>0}$ be two definable functions.
	Assume that the following condition is satisfied:
	$$
	\forall x \in C, \exists \delta>0, \forall x' \in C,\ |x'-x| < \delta \Rightarrow \varphi(x') \geq \psi(x) \text{.}
	$$
	Then the inequality $\inf \varphi(C)>0$ holds true.
\end{lemma}
\begin{proof}
	Assume for contradiction that $\inf\varphi(C)=0$.
	For every $\varepsilon > 0$, let $D_{\varepsilon}$ be the closure of $\varphi^{-1}([0,\varepsilon])$ .
	The family $(D_{\varepsilon})_{\varepsilon>0}$ is a definable filtered collection of nonempty closed subsets of $C$.
	Therefore, there exists a point $x$ in their intersection by \cite[Remark 5.6]{FKK}. 
	Let $0 < \varepsilon < \psi(x)$ and choose $\delta>0$ so that the condition in the lemma is satisfied. 
	Pick $x_0 \in \varphi^{-1}([0,\varepsilon])$ such that $|x_0-x|<\delta$.
	We have $\varepsilon \geq \varphi(x_0) \geq \psi(x) > \varepsilon$.
	It is a contradiction.
\end{proof}

\begin{definition}
	Consider an expansion of a densely linearly ordered abelian group $\mathcal M=(M,<,+,0,\ldots)$.
	Let $C$ and $P$ be definable sets.
	Let $f: C \times P \rightarrow M$ be a definable function.
	The function $f$ is \textit{equi-continuous} with respect to $P$ if the following condition is satisfied:
	$$
	\forall \varepsilon>0, \ \forall x \in C, \ \exists \delta >0, \ \forall p \in P, \ \forall x' \in C,\ \  |x-x'|< \delta \Rightarrow |f(x,p)-f(x',p)|< \varepsilon\text{.}
	$$
	The function $f$ is \textit{uniformly equi-continuous} with respect to $P$ if the following condition is satisfied:
	$$
	\forall \varepsilon>0, \ \exists \delta >0, \ \forall p \in P, \ \forall x, x' \in C,\ \  |x-x'|< \delta \Rightarrow |f(x,p)-f(x',p)|< \varepsilon\text{.}
	$$
%
\end{definition}

\begin{proposition}\label{prop:equi-cont}
	Consider a definably complete expansion of an ordered group $\mathcal M=(M,<,+,0,\ldots)$.
	Let $C$ and $P$ be definable sets.
	Let $f: C \times P \rightarrow M$ be a definable function.
	Assume that $C$ is closed and bounded.
	Then $f$ is equi-continuous with respect to $P$ if and only if it is uniformly equi-continuous with respect to $P$.
\end{proposition}
\begin{proof}
	A uniformly equi-continuous definable function is always equi-continuous.
	We prove the opposite implication.
	
	Take a positive $c \in M$.
	Consider the definable function $\varphi:C \times M^{>0} \rightarrow M^{>0}$ given by 
	$$ \varphi(x,\varepsilon)=\sup\{0<\delta<c\;|\;\forall p \in P, \ \forall x' \in C,\  |x-x'|< \delta \Rightarrow |f(x,p)-f(x',p)|< \varepsilon\}\text{.}$$
	Since $f$ is equi-continuous with respect to $P$, we have $\varphi(x,\varepsilon)>0$ for all $x \in C$ and $\varepsilon>0$.
	Fix arbitrary $x \in C$ and $\varepsilon>0$.
	We also fix an arbitrary point $x' \in C$ with $|x'-x|<\frac{1}{2}\varphi(x,\frac{\varepsilon}{2})$.
	We have $|f(x',p)-f(x,p)|<\frac{\varepsilon}{2}$ for all $p \in P$ by the definition of $\varphi$.
	
	For all $y \in C$ with $|x'-y|<\frac{1}{2}\varphi(x,\frac{\varepsilon}{2})$, we have $|x-y| \leq |x-x'|+|x'-y| < \varphi(x,\frac{\varepsilon}{2})$.
	We get $|f(y,p)-f(x,p)|<\frac{\varepsilon}{2}$ for all $p \in P$ by the definition of $\varphi$.
	We finally obtain $|f(y,p)-f(x',p)| \leq |f(x',p)-f(x,p)|+|f(y,p)-f(x,p)|<\varepsilon$ for all $p \in P$.
	It means that $\varphi(x',\varepsilon) \geq \frac{1}{2}\varphi(x,\frac{\varepsilon}{2})$ whenever $|x'-x|<\frac{1}{2}\varphi(x,\frac{\varepsilon}{2})$.
	Apply Lemma \ref{lem:prepre} to the definable functions $\varphi(\cdot,\varepsilon)$ and $\frac{1}{2}\varphi(\cdot,\frac{\varepsilon}{2})$ for a fixed $\varepsilon>0$.
	We have $\inf \varphi(C,\varepsilon)>0$.
	
	For any $\varepsilon>0$, set $\delta=\inf \varphi(C,\varepsilon)$.
	For any $p \in P$ and $x,x' \in C$, we have $|f(x,p)-f(x',p)|< \varepsilon$ whenever $|x-x'|< \delta$ by the definition of $\varphi$.
	It means that $f$ is uniformly equi-continuous.
\end{proof}

It is well known that a continuous function defined on a compact set is uniformly continuous.
The following corollary claims that a similar assertion holds true for a definable function defined on a definable, closed bounded set.
\begin{corollary}\label{cor:uniform}
	Consider a definably complete expansion of an ordered group $\mathcal M=(M,<,+,0,\ldots)$.
	Let $C$ be a definable, closed and bounded set.
	A definable continuous function $f: C \rightarrow M$ is uniformly continuous.
\end{corollary}
\begin{proof}
	Let $P$ be a singleton.
	Apply Proposition \ref{prop:equi-cont} to the function $g:C \times P \rightarrow M$ defined by $g(x,p)=f(x)$.
\end{proof}

\section{Definable directed sets}\label{sec:directed}
Our target in this section is to demonstrate \cite[Theorem 8]{GTW} in the definably complete locally o-minimal setting.
We first recall the definition of downward directed sets.

\begin{definition}\label{def:directed}
	Recall that a preorder is a transitive and reflexive binary relation.
	Let $\Omega$ be a set and $\preceq$ be a preorder on it.
	For any subset $\Omega'$, the notation $\preceq_{\Omega'}$ is a preorder on $\Omega'$ such that $u \preceq_{\Omega'} v$ if and only if $u \preceq v$ for all $u,v \in \Omega'$.
	The preorder $\preceq_{\Omega'}$ is called the \textit{restriction of $\preceq$ to $\Omega'$}.
	We simply denote $\preceq_{\Omega'}$ by $\preceq$ when this abuse of notations do not confuse readers.
	
	Let $(\Omega',\preceq')$ and $(\Omega,\preceq)$ be preordered sets.
	Given $S \subseteq \Omega$, a map $\gamma:\Omega' \to \Omega$ is \textit{downward cofinal for $S$} if, for every $u \in S$, there exists $v \in \Omega'$ such that $w \preceq' v$ implies $\gamma(w) \preceq u$.
	We say that $\gamma$ is \textit{downward cofinal} if it is downward cofinal for $\Omega$.
	
	A \textit{downward directed} set $(\Omega,\preceq)$ is a nonempty set with a preorder $\preceq$ such that, for any finite subset $\Omega'$ of $\Omega$, there exists $v \in \Omega$ satisfying the relation $v \preceq u$ for all $u \in \Omega'$.
	The dual notion of upward directed set is defined analogously.
	We mainly consider a downward directed set in this paper.
	We simply call it a directed set.
	
	Let $\mathcal M=(M,<,\ldots)$ be an expansion of a dense linear order without endpoints.
	A \textit{definable preordered set} is a definable set $\Omega \subseteq M^n$ together with a definable preorder $\preceq$ on $\Omega$.
	A \textit{definable direct set} is a definable preordered set which is directed.
%
\end{definition}

We give several examples of directed sets.
\begin{lemma}\label{lem:obv1}
	Consider a set $\Omega$ together with a linear order $\preceq$ on $\Omega$.
	The pair $(\Omega, \preceq)$ is a directed set.
\end{lemma}
\begin{proof}
	In fact, for any finite subset $\Omega'$ of $\Omega$, we have $\min_{u \in \Omega'} u \preceq u$ for all $u \in \Omega'$.
\end{proof}

\begin{example}\label{ex:1}
Let $\mathcal M=(M,<,+,0,\ldots)$ be an expansion of an ordered group.
Let $M^{>0}$ be the set of positive elements in $M$.
We equip $M^{>0} \times M^{>0}$ with the definable preorder given by $$(s,t) \unlhd (s',t') \Leftrightarrow s \leq s' \text{ and } t \geq t'.$$
Note that $(M^{>0} \times M^{>0}, \unlhd)$ is a definable directed set.	
\end{example}

The following awkward linear order may seem to be tricky but it is useful.
\begin{example}\label{ex:awk}
	Let $\mathcal M=(M,<,+,0,\ldots)$ be an expansion of an ordered group.
	We introduce a linear order $\awk$ on $M^n$ called \textit{awkward linear order}.
	For any $x=(x_1,\ldots, x_n)$, we denote $(|x_1|,\ldots, |x_n|)$ by $|x|$. 
	Let $\mathfrak  S_n$ be the set of permutations of $\{1,2\ldots,n\}$.
	Since it is a finite set, we can define a linear order $\leq_{\mathfrak  S_n}$ on $\mathfrak S_n$.
	We fix it.
	For any $\sigma \in \mathfrak  S_n$ and $x=(x_1,\ldots, x_n) \in M^n$, we set $\sigma(x)=(x_{\sigma(1)}, \ldots, x_{\sigma(n)})$. 
	Define the sign map $\sgn:M \to \mathbb Z$ by $$\sgn(x)=\left\{\begin{array}{lc} 1 & \text{if } x>0,\\ 0 & \text{if }x=0,\\ -1 &\text{if }x<0. \end{array}\right.$$
	We set $\sgn(x)=(\sgn(x_1),\ldots, \sgn(x_n))$ for any $x=(x_1,\ldots, x_n) \in M^n$.
	For any $x \in M^n$, take the permutation $\sigma_x$ which is the minimum of the permutations $\sigma \in \mathfrak  S_n$ such that $|x_{\sigma(1)}| \geq |x_{\sigma(2)}| \geq \cdots \geq |x_{\sigma(n)}|$ with respect to the order $\leq_{\mathfrak  S_n}$.
	
	We are now ready to define the awkward linear order $\awk$.
	Let $x$ and $y$ be elements in $M^n$.
	We define $x \awk y$ if and only if one of the following conditions is satisfied:
	\begin{itemize}
		\item The element $|\sigma_x(x)|$ is smaller than $|\sigma_y(y)|$ with respect to the lexicographic order induced from the given linear order $<$;
		\item We have the equality $|\sigma_x(x)|=|\sigma_y(y)|$ and $\sigma_x$ is smaller than $\sigma_y$ with respect to the order $\leq_{\mathfrak S_n}$;
		\item We have the equalities $|\sigma_x(x)|=|\sigma_y(y)|$ and $\sigma_x = \sigma_y$. Furthermore, $\sgn(\sigma_x(x))$ is not greater than $\sgn(\sigma_y(y))$ with respect to the  lexicographic order on $\mathbb Z^n$.
	\end{itemize}
	It is obvious that the binary relation $\awk$ is a definable linear order.
	For any definable subset $X$ of $M^n$, the restriction of $\awk$ to $X$ is called the \textit{awkward linear order} on $X$ and also denoted by $\awk$.
\end{example}

We give several lemmas.
We do not repeat the proof when the proof in \cite{GTW} for o-minimal structures is also applicable to the definably complete locally o-minimal case.
\begin{lemma}\label{lem:decom_cont}
	Let $\mathcal M=(M,<,\ldots)$ be a definably complete locally o-minimal structure.
	Let $X$ be a definable set and $f:X \to M^n$ be a definable map.
	The definable set $X$ is partitioned into finitely many disjoint definable subsets $X_1, \ldots X_l$ such that the restriction of $f$ to $X_i$ is continuous for each $1 \leq i \leq l$.
\end{lemma}
\begin{proof}
	It is easy to prove the lemma using \cite[Proposition 2.7(2)]{FKK} by induction on $\dim X$.
	We omit the details.
\end{proof}

\begin{definition}
	A family of sets $\mathcal S$ has the \textit{finite intersection property} if, for any finite subfamily $\mathcal T$ of $\mathcal S$, we have $\bigcap_{S \in \mathcal T} S \neq \emptyset$.
\end{definition}

Let $(M,+,0,<)$ be an ordered group.
For any $u=(u_1,\ldots,u_n ) \in M^n$, the norm $\max\{|u_i|\;|\;1 \leq i \leq n\}$ is denoted by $\|u\|$.
For any $u \in M^n$ and $\varepsilon >0$, we set $\mathcal B_n(u,\varepsilon):=\{x \in M^n\;|\; \|x-u\|<\varepsilon\}$.
It is an open box.

The following and next lemmas are ley lemmas in this section.
It is a counterpart of \cite[Lemma 10]{GTW}.
\begin{lemma}\label{lem:key1}
	Let $\mathcal M=(M,<,+,0,\ldots)$ be a definably complete locally o-minimal expansion of an ordered group.
	Let $\mathcal S=\{S_u\;|\; u \in \Omega\}$ be a definable family of subsets of $M^n$ with the finite intersection property.
	There exist definable maps $f:\Omega \to M^n$, $\delta:\Omega \to M^{>0}$ and $\eta:\Omega \to M^{>0}$ satisfying the following conditions:
	\begin{enumerate}
		\item[(1)] $f(u) \in S_u$;
		\item[(2)] For every $u,v \in \Omega$, the inequalities $\frac{1}{2}\delta(u)<\delta(v)$, $\frac{1}{2}\eta(u)<\eta(v)$ and $\|f(u)-f(v)\|< \min\{\frac{1}{2}\delta(u),\frac{1}{2}\eta(u)\}$ imply $f(u) \in S_v$.
	\end{enumerate}
\end{lemma}
\begin{proof}
	The following claim is obvious and we omit its proof.
	\medskip
	
	\textbf{Claim.} Let $\mathcal S'=\{S'_u\;|\; u \in \Omega\}$ be another definable family of subsets of $M^n$ with the finite intersection property such that $S'_u \subseteq S_u$ for all $u \in \Omega$.
	If the lemma holds true for the family $\mathcal S'$, the lemma holds true for the family $\mathcal S$.
	\medskip
	
	Using the above claim, we reduce to the case in which there exists a nonnegative integer $m$ such that $\dim(\bigcap_{u \in \mathcal F}S_u)=m$ for each finite set $\mathcal F$ of $\Omega$.
	Set $m=\min\{\dim(\bigcap_{u \in \mathcal F}S_u)\;|\;\mathcal F \subseteq \Omega: \text{finite subset} \}$.
	There exists a finite subset $\mathcal F_0$ of $\Omega$ such that $\dim(\bigcap_{u \in \mathcal F_0}S_u)=m$.
	Set $C=\bigcap_{u \in \mathcal F_0}S_u$ and $\mathcal S'=\{C \cap S_u\;|\; u \in \Omega\}$.
	It is obvious that $\mathcal S'$ enjoys the finite intersection property.
	Considering $\mathcal S'$ instead of $\mathcal S$, we may assume that $\dim(\bigcap_{u \in \mathcal F}S_u)=m$ for each finite set $\mathcal F$ of $\Omega$ by the claim.
	
	We may further assume that there exist a coordinate projection $\pi: M^n \to M^m$ and a $\pi$-quasi-special submanifold $X$ such that $S_u$ is contained in $X$ for every $u \in \Omega$.
	We justify this assumption.
	Fix a point $u_0 \in \Omega$.
	We can decompose $S_{u_0}$ into finitely many disjoint quasi-special submanifolds $C_1, \ldots, C_l$ by \cite[Proposition 2.11]{FKK}.
	We demonstrate that there exists $1 \leq q \leq l$ such that $\dim(C_q \cap \bigcap_{u \in \mathcal F}S_u)=m$ for each finite subset $\mathcal F$ of $\Omega$.
	Assume for contradiction that there does not exist such a $q$.
	For each $1 \leq i \leq l$, there exists a finite subset $\mathcal F_i$ such that $\dim(C_i \cap \bigcap_{u \in \mathcal F_i}S_u)<m$.
	We then have 
	\begin{align*}
	m &= \dim\left(S_{u_0} \cap \bigcap_{i=1}^l\bigcap_{u \in \mathcal F_i}S_u\right) =  \dim\left(\left(\bigcup_{i=1}^l C_i\right) \cap \bigcap_{i=1}^l\bigcap_{u \in \mathcal F_i}S_u\right) \\
	& \leq \dim \left(\bigcup_{i=1}^l C_i \cap \bigcap_{u \in \mathcal F_i}S_u\right)<m
	\end{align*}
	by \cite[Proposition 2.8(2)(5)]{FKK}.
	It is a contradiction.
	Let $\pi$ be the coordinate projection such that $C_q$ is a $\pi$-quasi-special submanifold.
	Set $X=C_q$.
	Consider the definable family $\mathcal S''=\{X \cap S_u\;|\; u\ \in \Omega\}$.
	It is obvious $\mathcal S''$ enjoys the finite intersection property.
	We have justified the assumption by Claim.
	
	We consider two separate cases, that is, the case in which $m=n$ and the case in which $m<n$.
	We first consider the case in which $m=n$.
	For any $u \in \Omega$, the definable set $S_u$ has a nonempty interior in this case.
	Let $\rho_1: \Omega \times X \times M^{>0} \to \Omega$, $\rho_2: \Omega \times X \times M^{>0} \to X$ and  $\rho_3: \Omega \times X \times M^{>0} \to M^{>0}$ be the projections.
	Set 
	\begin{align*}
	Z &= \{(u,x,\varepsilon) \in \Omega \times X \times M^{>0}\;|\; \mathcal B_m(x,\varepsilon) \subseteq S_u\}.
	\end{align*}
	By the assumption, we have $\rho_1^{-1}(u) \cap Z \neq \emptyset$ for any $u \in \Omega$.
	There exists a definable map $\tau: \Omega \to Z$ such that the composition $\rho_1 \circ \tau$ is the identity map on $\Omega$ by Lemma \ref{lem:definable_choice}.
	Set $f=\rho_2 \circ \tau$ and $\delta=\eta=\rho_3 \circ \tau$.
	These there maps satisfy the requirements in the lemma.
	In fact, it is obvious that $f(u) \in S_u$ for any $u \in \Omega$.
	If $\frac{1}{2}\delta(u)<\delta(v)$ and $\|f(u)-f(v)\|< \frac{1}{2}\delta(u)$, we have $\|f(u)-f(v)\|<\delta(v)$.
	Since $S_v$ contains an open box $\mathcal B_m(f(v),\delta(v))$, we have $f(u) \in S_v$.
	
	We next consider the case in which $m<n$.
	The definable set $X$ is a $\pi$-quasi-special submanifold.
	We may assume that $\pi$ is the coordinate projection onto the first $m$ coordinates by permuting the coordinates if necessary.
	We want to show that, for each $u \in \Omega$, there exists an open box $B'$  such that $S_u \cap B'=X \cap B'$ and $S_u \cap B'$ is the graph of a definable continuous map defined on $\pi(B')$.	
	By \cite[Proposition 2.7(10)]{FKK}, for each $u \in \Omega$, there exists $y \in S_u$ such that $\dim (B \cap S_u)=m$ for any open box containing the point $y$.
	By the definition of quasi-special submanifolds, if we choose a sufficiently small open box $B$ with $y \in B$, the intersection $X \cap B$ is the graph of a definable continuous function $\tau_B$ defined on $\pi(B)$.
	For any $z \in \pi(S_u \cap B)$, the set $\pi^{-1}(z) \cap (S_u \cap B)$ is a singleton.
	We have $\dim(\pi(S_u \cap B))=m$ by \cite[Proposition 2.7(1),(9)]{FKK} because $\dim(S_u \cap B)=m$.
	It means that $\pi(S_u \cap B)$ has a nonempty open interior by the definition of dimension.
	Take a point $w$ in the interior of $\pi(S_u \cap B)$ and set $x=(w,\tau_B(w))$.
	Take a sufficiently small open box $C$ in $M^m$ containing the point $w$ and contained in the interior of $\pi(S_u \cap B)$.
	We set $B'=\pi^{-1}(C) \cap B$.
	We have $X \cap B'_u=S_u \cap B'$.
	We have demonstrated that, for each $u \in \Omega$, there exists an open box $B$ such that $S_u \cap B=X \cap B$ and $S_u \cap B$ is the graph of a definable continuous map defined on $\pi(B)$.
	
	Let $\rho'_1:\Omega \times X \times M^{>0} \times M^{>0} \to \Omega$, 
	$\rho'_2:\Omega \times X \times M^{>0} \times M^{>0} \to X$, $\rho'_3:\Omega \times X \times M^{>0} \times M^{>0} \to M^{>0}$ and $\rho'_4:\Omega \times X \times M^{>0} \times M^{>0} \to M^{>0}$ be the projections.
	The projections $\rho'_3$ and $\rho'_4$ are the projections onto the second coordinate to the last and onto the last coordinate, respectively.
	Let $\pi':M^n \to M^{n-m}$ be the coordinate projection forgetting the first $m$ coordinates.
	We set $\mathcal B(x,\varepsilon_1,\varepsilon_2):=\mathcal B_m(\pi(x),\varepsilon_1) \times \mathcal B_{n-m}(\pi'(x),\varepsilon_2)$ for any $x \in M^n$, $\varepsilon_1>0$ and $\varepsilon_2>0$.
	Set 
	\begin{align*}
	Z'&=\{(u,x,\varepsilon_1,\varepsilon_2) \in \Omega \times X \times M^{>0} \times M^{>0}\;|\; x \in S_u\\
&\qquad  S_u \cap \mathcal B(x,\varepsilon_1,\varepsilon_2) = X \cap \mathcal B(x,\varepsilon_1,\varepsilon_2) \text{ and } S_u \cap \mathcal B(x,\varepsilon_1,\varepsilon_2) \text{ is the graph of }\\
	&\qquad \text{a definable continuous map defined on } \mathcal B_m(\pi(x),\varepsilon_1)\}.
	\end{align*}
	We have demonstrated that the intersection $(\rho'_1)^{-1}(u) \cap Z'$ is not empty for each $u \in \Omega$ in the previous paragraph.
	We can find a definable map $\tau':\Omega \to Z'$ such that the composition $\rho'_1 \circ \tau'$ is the identity map on $\Omega$ by Lemma \ref{lem:definable_choice}.
	Set $f=\rho'_2 \circ \tau'$, $\delta = \rho'_3 \circ \tau'$ and $\eta=\rho'_4 \circ \tau'$. 
	They are the maps we are looking for.
	It is obvious that $f(u) \in S_u$.
	In particular, we have $f(u) \in X$.
	Assume that the inequalities $\frac{1}{2}\delta(u)<\delta(v)$, $\frac{1}{2}\eta(u)<\eta(v)$ and $\|f(u)-f(v)\|< \min\{\frac{1}{2}\delta(u),\frac{1}{2}\eta(u)\}$ hold true.
	We have $\|\pi(f(u))-\pi(f(v))\|<\frac{1}{2}\delta(u)<\delta(v)$ and $\|\pi'(f(u))-\pi'(f(v))\|<\frac{1}{2}\eta(u)<\eta(v)$ by the definition of the norm.
	It means that $f(u)$ is contained in the box $\mathcal B(f(v),\delta(v),\eta(v))$.
	We get $f(u) \in X \cap \mathcal B(f(v),\delta(v),\eta(v))$.
	On the other hand, we have $S_v \cap \mathcal B(f(v),\delta(v),\eta(v))=X \cap \mathcal B(f(v),\delta(v),\eta(v))$.
	It implies that $f(u) \in S_v$.
\end{proof}

We next prove the second key lemma.
\begin{lemma}\label{lem:main}
	Let $\mathcal M=(M,<,+,0,\ldots)$ be a definably complete locally o-minimal expansion of an ordered group.
	Let $\mathcal S=\{S_u\;|\; u \in \Omega\}$ be a definable family of subsets of $M^n$ with the finite intersection property.
	There exist a definable set $\Omega' \subseteq M^N$ and a definable bijection $h:\Omega \to \Omega'$ satisfying the following conditions:
	\begin{enumerate}
		\item[(1)] For every $u \in \Omega'$, there exists $\varepsilon>0$ such that $\bigcap_{v \in \Omega', \|v-u\|<\varepsilon} S_{h^{-1}(v)} \neq \emptyset$.
		\item[(2)] For every definable subset $B$ of $\Omega'$ which is closed and bounded in $M^N$, there exists $\varepsilon>0$ such that $\bigcap_{v \in \Omega', \|v-u\|<\varepsilon} S_{h^{-1}(v)}  \neq \emptyset$ for every $u \in B$.
	\end{enumerate}
\end{lemma}
\begin{proof}
	Let $f$, $\delta$ and $\eta$ be definable maps given in Lemma \ref{lem:key1}.
	We first consider the case in which these three maps are continuous.
	In this case, we set $\Omega'=\Omega$ and $h$ is the identity map.
	We demonstrate that the conditions (1) and (2) are satisfied.
	If we apply the condition (2) to the singleton $\{u\}$, we get the condition (1).
	We have only to demonstrate the condition (2).
	
	By the max-min theorem in \cite{M}, there are $d_B$ and $e_B$ such that $d_B \leq \frac{1}{2}\delta(u)$ and $e_B \leq \frac{1}{2}\eta(u)$ for all $u \in B$.
	The restrictions of $f$, $\delta$ and $\eta$ to $B$ is uniformly continuous by Corollary \ref{cor:uniform}.
	We can take $\delta_{f,B}>0$, $\delta_{\delta,B}>0$ and $\delta_{\eta,B}>0$ such that the following conditions are satisfied:
	\begin{itemize}
		\item $\|u-v\| < \delta_{f,B} \Rightarrow \|f(u)-f(v)\|<\min\{d_B, e_B\}$;
		\item $\|u-v\| < \delta_{\delta,B} \Rightarrow |\delta(u)-\delta(v)|<d_B$;
		\item $\|u-v\| < \delta_{\eta,B} \Rightarrow |\eta(u)-\eta(v)|<e_B$.
	\end{itemize}
	Set $\varepsilon=\min\{\delta_{f,B},\delta_{\delta,B},\delta_{\eta,B}\}$.
	We fix $u \in B$.
	Take an arbitrary point $v \in \Omega$ with $\|u-v\| < \varepsilon$.
	We easily have $\frac{1}{2} \delta(u)<\delta(v)$, $\frac{1}{2} \eta(u)<\eta(v)$ and $\|f(u)-f(v)\|<\min\{\frac{1}{2} \delta(u),\frac{1}{2} \eta(u)\}$.
	By Lemma \ref{lem:key1}, we have $f(u) \in S_v$.
	We have shown that $f(u) \in \bigcap_{v \in \Omega', \|v-u\|<\varepsilon} S_{h^{-1}(v)}$.
	
	We next consider the case in which $f$, $\delta$ and $\eta$ are not necessarily continuous.
	There exists a finite partition $\Omega=\Omega_1 \cup \ldots \cup \Omega_l$ such that the restrictions $f$, $\delta$ and $\eta$ to $\Omega_i$ are continuous for all $1 \leq i \leq l$ by Lemma \ref{lem:decom_cont}.
	Take a positive element $c$.
	We set $\Omega'=\bigcup_{i=1}^l \{i\cdot c\} \times \Omega_i$.
	Consider the definable bijection $h_i:\Omega \to \Omega'$ defined by $h(x)=(i \cdot c,x)$ when $x \in \Omega_i$.
	The compositions $f \circ h^{-1}$, $\delta \circ h^{-1}$ and $\eta \circ h^{-1}$ are continuous.
	The proof proceeds similarly to the case in which $f$, $\delta$ and $\eta$ are continuous.
	We omit the details.
\end{proof}

\begin{definition}
	Let $(\Omega,\preceq)$ be a directed set.
	We say that $\Omega' \subseteq \Omega$ is $\preceq$-bounded in $\Omega$ if there exists $v \in \Omega$ such that $v \preceq u$ for all $u \in \Omega'$.
	We write $v \preceq \Omega'$ in this case.
\end{definition}

\begin{corollary}\label{cor:bounded}
	Let $\mathcal M=(M,<,+,0,\ldots)$ be a definably complete locally o-minimal expansion of an ordered group.
	Any definable directed set $(\Omega', \preceq')$ is definably preorder isomorphic to a directed set $(\Omega,\preceq)$ satisfying the following:
	\begin{enumerate}
		\item[(1)] For all $u \in \Omega$, there exists an $\varepsilon>0$ such that $\mathcal B(x,\varepsilon) \cap \Omega$ is $\preceq$-bounded in $\Omega$;
		\item[(2)] For any definable subset $B \subseteq \Omega$ which is bounded and closed in the ambient space $M^N$ of $\Omega$, there exists an $\varepsilon>0$ such that $\mathcal B(u,\varepsilon) \cap \Omega$ is $\preceq$-bounded in $\Omega$ for every $u \in B$.
	\end{enumerate}
\end{corollary}
\begin{proof}
	Consider the definable family $\mathcal S=\{S_u\;|\;u \in \Omega'\}$ defined by $S_u=\{v \in \Omega'\;|\;v \preceq' u\}$.
	Apply Lemma \ref{lem:main} to this family.
\end{proof}

\begin{lemma}\label{lem:basic2}
	Let $\mathcal M=(M,<,+,0,\ldots)$ be a definably complete expansion of an ordered group.
	Let $(\Omega,\preceq)$ be a definable directed set and $S \subseteq \Omega$ be a bounded definable set.
	Assume that there exists $\varepsilon_0>0$ such that, for all $u \in S$, $B(u,\varepsilon_0) \cap S$ is $\preceq$-bounded in $\Omega$.
	Then $S$ is $\preceq$-bounded in $\Omega$.
\end{lemma}
\begin{proof}
	We use \cite[Lemma 13]{GTW}.
	It is proved for an o-minimal expansion of an ordered group, but it holds true for any expansion of a densely linearly ordered abelian group by the literally same proof.
	Our lemma is a counterpart of \cite[Lemma 14]{GTW}.
	In the proof of \cite[Lemma 14]{GTW}, the definable set
	$$H=\{\varepsilon>0\;|\; \forall u \in S, B(u,\varepsilon) \cap S \text{ is } \preceq\text{-bounded in }\Omega\}$$
	is considered.
	Set $r=\sup H$, which exists in $M \cup \{\infty\}$ because $\mathcal M$ is definably complete.
	It is obvious that $u \in H$ for any $0<u<r$.
	We have only to demonstrate $\sup H=\infty$.
	It is proven in the same manner as \cite[Lemma 14]{GTW}.
\end{proof}

Corollary \ref{cor:bounded} and Lemma \ref{lem:basic2} together yield the following corollary:
\begin{corollary}\label{cor:bounded2}
	Let $\mathcal M=(M,<,+,0,\ldots)$ be a definably complete locally o-minimal expansion of an ordered group.
	Any definable directed set $(\Omega', \preceq')$ is definably preorder isomorphic to a directed set $(\Omega,\preceq)$ such that any definable subset $\Omega$ which is bounded and closed in the ambient space $M^N$ of $\Omega$ is $\preceq$-bounded in $\Omega$.
\end{corollary}

We recall $\myds$-sets introduced in \cite{DMS}.
\begin{definition}[$\myds$-sets]
	Consider an expansion of a linearly ordered structure $\mathcal M=(M,<,0,\ldots)$.
	A \textit{parameterized family} of definable sets $\{X{\langle x \rangle}\}_{x \in S}$ is the family of the fibers of a definable set; that is, there exists a definable set $\mathcal X$ with $X{\langle x \rangle}=\mathcal X_x$ for all $x$ in a definable set $S$.
	A parameterized family $\{X{\langle r,s\rangle}\}_{r>0,s>0}$ of closed, bounded and definable subsets $X{\langle r,s\rangle}$ of $M^n$ is called a \textit{$\myds$-family} if $X{\langle r,s\rangle} \subseteq X{\langle r',s\rangle}$ and  $X{\langle r,s'\rangle } \subseteq X{\langle r,s\rangle}$ whenever $r < r'$ and $s < s'$.
	Note that $X{\langle r,s\rangle}$ is not necessarily strictly contained in $X{\langle r',s\rangle}$.
	It is the same for the inclusion $X{\langle r,s'\rangle } \subseteq X{\langle r,s\rangle}$.
	A definable subset $X$ of $M^n$ is a \textit{$\myds$-set} if $X = \displaystyle\bigcup_{r>0,s>0} X{\langle r,s \rangle}$ for some $\myds$-family $\{X{\langle r,s\rangle}\}_{r>0,s>0}$.
%
\end{definition}

\begin{lemma}\label{lem:dsigma}
Any set definable in a definably complete locally o-minimal expansion of an ordered group is $\myds$.
\end{lemma}
\begin{proof}
	The same proof as \cite[Lemma 17]{GTW} works using \cite[Proposition 2.8(8)]{FKK} instead of the counterpart in the o-minimal setting.
\end{proof}

\begin{theorem}\label{thm:downward}
		Let $\mathcal M=(M,<,+,0,\ldots)$ be a definably complete locally o-minimal expansion of an ordered group.
		Let $(\Omega,\preceq)$ be a definable directed set.
		There exists a definable downward cofinal map $\gamma:(M^{>0}\times M^{>0}, \unlhd) \to (\Omega,\preceq)$.
\end{theorem}
\begin{proof}
	A similar proof to that of \cite[Theorem 8]{GTW} works.
	We use Corollary \ref{cor:bounded2} and Lemma \ref{lem:dsigma} instead of Corollary 15 and Lemma 17 of \cite{GTW}, respectively.
\end{proof}

\section{Definable topological spaces}\label{sec:compact}
We propose a definition of definable compactness of definable topological spaces.
We first recall the definition of definable topological spaces.

\begin{definition}\label{def:topology}
	Let $\mathcal M=(M,<,\ldots)$ be an expansion of a dense linear order.
	A topological space $(X,\tau)$ is a \textit{definable topological space} if $X$ is a definable subset and there exists a definable family of sets $\mathcal B$ which is an open base for the topology $\tau$.
	We call the open base $\mathcal B$ a \textit{definable open base} of the topology $\tau$.
	We define a \textit{definable basis of neighborhoods} of $X$ at a point $x \in X$ in the same manner.
	The linear order $<$ induces a topology on $M$ called the \textit{order topology}.
	The Cartesian product $M^n$ equips the product topology of the order topology.
	Any definable subset $X$ of $M^n$ has the relative topology induced from the product topology of $M$.
	It is called the \textit{affine topology} on $X$ in this paper.
\end{definition}

The following example illustrates that definable  neighborhoods of a point generate a definable directed set.
\begin{example}\label{ex:basis}
	Consider a definably complete locally o-minimal expansion of an ordered group $\mathcal M=(M,<,+,0,\ldots)$.
	Let $(X,\tau)$ be a definable topological space with a definable open base $\mathcal B=\{A_u\;|\; u \in \Omega\}$.
	Let $X$ be an arbitrary point in $X$.
	Set $\Omega_x=\{u \in \Omega\;|\; x \in A_u\}$.
	We consider the binary relation $\preceq_{\mathcal B,x}$ on $\Omega_x$ defined by $u \preceq_{\mathcal B,x} v \Leftrightarrow A_u \subseteq A_v$.
	The pair $(\Omega_x, \preceq_{\mathcal B,x})$ is a definable directed set.
\end{example}

The following theorem is a corollary of Theorem \ref{thm:downward} and the above example.
It is a counterpart of \cite[Theorem 39]{GTW}.
\begin{theorem}\label{thm:countability}
	Consider a definably complete locally o-minimal expansion of an ordered group $\mathcal M=(M,<,+,0,\ldots)$.
	Let $(X,\tau)$ be a definable topological space with a definable open base $\mathcal B$.
	Let $x$ be an arbitrary point in $X$.
	Then the following assertions hold true:
	\begin{enumerate}
		\item[(1)]   There exists a definable basis of neighborhoods of $x$ of the form $$\{A_{(s,t)}\;|\;s,t>0\} \subseteq \mathcal B$$ such that, for any neighborhood $A$ of $x$, there exists $(s_A,t_A) \in M^{>0} \times M^{>0}$ such that $A_{(s,t)} \subseteq A$ whenever $(s,t) \unlhd (s_A,t_A)$.
		\item[(2)] There exists a definable basis of open neighborhoods of $x$ of the form $$\{A'_{(s,t)}\;|\;s,t>0\}$$ satisfying $A'_{(s',t')} \subseteq A'_{(s,t)}$ whenever $(s',t') \unlhd (s,t)$.
		\item[(3)] If $(X,\tau)$ is regular, there exists a definable basis of closed neighborhoods of $x$ of the form $$\{A''_{(s,t)}\;|\;s,t>0\}$$ satisfying $A''_{(s',t')} \subseteq A''_{(s,t)}$ whenever $(s',t') \unlhd (s,t)$.
	\end{enumerate}
\end{theorem}
\begin{proof}
	We first prove the assertion (1).
	Set $\mathcal B=\{A_u\;|\; u \in \Omega\}$.
	Let $(\Omega_x,\preceq_{\mathcal B,x})$ be the definable directed set defined in Example \ref{ex:basis}.
	There exists a definable downward cofinal map $\gamma:(M^{>0}\times M^{>0}, \unlhd) \to (\Omega_x,\preceq_{\mathcal B,x})$ by Theorem \ref{thm:downward}.
	Set $A_{(s,t)}=A_{\gamma(s,t)} \in \mathcal B$ for all $(s,t) \in M^{>0}\times M^{>0}$.
	The definable basis of neighborhoods of $x$ defined by $\{A_{(s,t)}\;|\; (s,t) \in M^{>0}\times M^{>0}\}$ obviously satisfies the requirement.
	
	The proof of the assertion (2) is easy.
	Let $\{A_{(s,t)}\;|\; (s,t) \in M^{>0}\times M^{>0}\}$ be the definable basis of neighborhoods of the point $x$ given in the assertion (1).
	Set $A'_{(s,t)}=\bigcup_{(s',t') \unlhd (s,t)} A_{(s',t')}$ for each $(s,t) \in M^{>0}\times M^{>0}$.
	It is a definable open neighborhood of the point $x$.
	The inclusion $A'_{(s',t')} \subseteq A'_{(s,t)}$ is clear when $(s',t') \unlhd (s,t)$.
	The definable family $\mathcal B'_x=\{A'_{(s,t)}\;|\; (s,t) \in M^{>0}\times M^{>0}\}$ is a definable basis of open neighborhoods of $x$.
	In fact, let $A$ be an arbitrary neighborhood of $x$ in $X$.
	Take $(s_A,t_A) \in M^{>0}\times M^{>0}$ satisfying the condition in the assertion (1).
	We obviously have $A'_{(s_A,t_A)}=\bigcup_{(s,t) \unlhd (s_A,t_A)}A_{(s,t)} \subseteq A$ because $A_{(s,t)} \subseteq A$ whenever $(s,t) \unlhd (s_A,t_A)$.
	
	We demonstrate the assertion (3).
	Let $\{A'_{(s,t)}\;|\;s,t>0\}$ be a definable basis of open neighborhoods of $x$ satisfying the conditions in (2).
	Set $A''_{(s,t)}:=\mycl_{\tau}(A'_{(s,t)})$ for each $s,t>0$.
	The definable set $A''_{(s,t)}$ is closed and it is a neighborhood of $x$.
	It is also obvious that $A''_{(s',t')} \subseteq A''_{(s,t)}$ whenever $(s',t') \unlhd (s,t)$.
	The remaining task is to demonstrate that $\{A''_{(s,t)}\;|\;s,t>0\}$ is a definable basis of neighborhoods of $x$.
	Take a neighborhood $A$ of $x$.
	We have only to show that there exist $s,t>0$ such that $A''_{s,t} \subseteq A$.
	We may assume that $A$ is open considering $\myint_{\tau}(A)$ instead of $A$ if necessary.
	Since $(X,\tau_X)$ is regular, there are open sets $C_1$ and $C_2$ such that $x \in C_1$, $X \setminus A \subseteq C_2$ and $C_1 \cap C_2 = \emptyset$.
	It implies that the complement $X \setminus C_2$ is a closed neighborhood of $x$ contained in $A$.
	By the definition of  $\{A'_{(s,t)}\;|\;s,t>0\}$, there exist $s,t>0$ such that $A'_{s,t} \subseteq X \setminus C_2$.
	Since $X \setminus C_2$ is closed, we have $A'_{s,t} \subseteq X \setminus C_2 \subseteq A$.
\end{proof}

Our next target is to demonstrate a counterpart of \cite[Corollary 44]{GTW} which gives equivalence conditions for a definable topological space to be definably compact in the o-minimal setting.
In our locally o-minimal setting, we need to introduce a slightly different concept.
\begin{definition}\label{def:limit}
	Let $\mathcal M=(M,<)$ be an expansion of a dense linear order without endpoints.
	Consider a definable topological space $(X,\tau)$.
	
	A \textit{definable net in $X$} is a definable map from a definable directed set $(\Omega,\preceq)$ to $X$.
	A definable net $\gamma: (\Omega,\preceq) \to (X,\tau)$ \textit{converges to $x \in X$} if, for every definable neighborhood $A$ of $x$, there exists $u_A \in \Omega$ such that $\gamma(u) \in A$ whenever $u \preceq u_A$.
	A \textit{definable subnet} $\gamma'$ of a definable net $\gamma$ is a definable net of the from $\gamma'=\gamma \circ f$ where $f:(\Omega',\preceq') \to (\Omega,\preceq)$ is a definable downward cofinal map.
	 	
	 A \textit{definable curve} is a definable map defined on an interval or its image.
	 Note that we do not require that a definable curve is continuous in this paper.
	We define the \textit{set of right convergences of a definable curve $\gamma:(a,b) \to (X,\tau)$} with $a \in M \cup \{-\infty\}$ and $b \in M \cup \{\infty\}$ as follows:
	\begin{align*}
	\convr(\gamma)&=\{x \in X\;|\; a<\forall t<b, \forall A: \text{definable neighborhood of }x, \\
	&\quad \gamma((t,b)) \cap A \neq \emptyset\}.
	\end{align*}
	The \textit{set of left convergences of $\gamma$} is defined similarly. 
	\begin{align*}
	\convl(\gamma)&=\{x \in X\;|\; a<\forall t<b, \forall A: \text{definable neighborhood of }x, \\
	&\quad \gamma((a,t)) \cap A \neq \emptyset\}.
\end{align*}
It is obvious that both $\convr(\gamma)$ and $\convl(\gamma)$ are definable sets and they are contained in the closure of $\gamma((a,b))$ in $X$.
\end{definition}

\begin{lemma}\label{lem:conv_formula}
	Consider a definably complete locally o-minimal expansion of an ordered group.
	The following assertions hold true:
	\begin{enumerate}
		\item[(1)]  Let $(X_i,\tau_i)$ be definable topological spaces for $i=1,2$.
		We consider the Cartesian product $X_1 \times X_2$ equipped with the product topology.
		Let $\gamma_i:(a,b) \to X_i$ be definable curves for $i=1,2$, and the definable curve $\gamma:(a,b) \to X$ is defined by $\gamma(t)=(\gamma_1(t),\gamma_2(t))$.
		Then, the equalities 
		\begin{align*}
		&\convl(\gamma)=\convl(\gamma_1) \times \convl(\gamma_2)\text{ and }\\ &\convr(\gamma)=\convr(\gamma_1) \times \convr(\gamma_2)
		\end{align*}
		 hold true.
		 \item[(2)] Let $(X,\tau_X)$ and $(Y,\tau_Y)$ be definable topological space and $f : (X,\tau_X) \to (Y,\tau_Y)$ be a definable continuous map.
		 Let $\gamma:(a,b) \to X$ be a definable curve.
		 Then, the equalities 
		 \begin{align*}
		 &f(\convl(\gamma)) \subseteq \convl(f \circ \gamma) \text{ and } f(\convr(\gamma)) \subseteq \convr(f \circ \gamma)
		 \end{align*}
		 hold true.
	\end{enumerate}
\end{lemma}
\begin{proof}
	The proofs are easy and straightforward.
	They are left to readers.
\end{proof}

We are now ready to prove the following theorem:

\begin{theorem}\label{thm:compact}
	Let $\mathcal M=(M,<,+,0,\ldots)$ be a definably complete locally o-minimal expansion of an ordered group.
	Let $(X,\tau)$ be a definable topological space.
	The following are equivalent:
	\begin{enumerate}
		\item[(1)] For any definable curve $\gamma:(a,b) \to X$ with $a \in M \cup \{-\infty\}$ and $b \in M \cup \{\infty\}$,  the definable sets $\convr(\gamma)$ and $\convl(\gamma)$ are not empty.
		\item[(2)] Every definable net in $X$ has a convergent definable subnet.
		\item[(3)] Every definable filtered collection of nonempty closed subsets of $X$ has a nonempty intersection.
	\end{enumerate}
\end{theorem}
\begin{proof}
	We first demonstrate that the condition (2) implies the condition (1).
	We only prove that $\convl(\gamma)$ is not empty.
	We can show that $\convr(\gamma) \neq \emptyset$ similarly.
	Since the order $\leq$ on the interval $(a,b)$  is a linear order, the pair $((a,b),\leq)$ is a definable directed set by Lemma \ref{lem:obv1}.
	Therefore, the definable curve $\gamma:(a,b) \to X$ is a definable net.
	By the condition (2), there exists a definable downward cofinal map $f:(\Omega, \preceq) \to ((a,b),\leq)$ such that the subnet $\gamma'=\gamma \circ f$ converges to a point $x \in X$.
	We fix an arbitrary definable neighborhood $A$ of the point $x$ and $t \in M$ satisfying $a<t<b$.
	There exists $u_A \in \Omega$ such that $\gamma(f(u)) \in A$ whenever $u \preceq u_A$ by the definition of net convergence.
	Since $f$ is downward cofinal, there exists $u_t \in \Omega$ such that $f(v) \leq t/2$ whenever $v \preceq u_t$.
	Take $u \in \Omega$ with $u \preceq u_A$ and $u \preceq u_t$.
	We have $\gamma(f(u)) \in A$ and $f(u)<t$.
	Set $z=f(u)$.
	We get $a<z<t$ and $\gamma(z) \in A$.
	It implies that $x \in \convl(\gamma)$.
	
	We show that the condition (3) implies the condition (2).
	The proof is almost the same as that of \cite[Corollary 44]{GTW}.
	Let $\gamma:(\Omega,\preceq) \to (X,\tau)$ be a definable net.
	Set $C_u=\gamma(\{u' \in \Omega\;|\; u' \preceq u\})$.
	The definable family $\mathcal C=\{\mycl_{\tau}(C_u)\;|\; u \in \Omega\}$ is a definable directed set of closed subsets of $X$.
	We can take a point $x \in \bigcap_{C \in \mathcal C}C$ by the condition (3).
	Take a definable basis of neighborhoods $\mathcal A=\{A_v\;|\; v \in \Sigma\}$ of $x$.
	Let $(\Sigma,\preceq_{\mathcal A})$ be the directed set defined by $v \preceq_{\mathcal A} v' \Leftrightarrow A_v \subseteq A_{v'}$.
	Since $x \in  \bigcap_{C \in \mathcal C}C$, the following sentence holds true:
	\begin{equation}
	\forall u \in \Omega,\ \forall v \in \Sigma, \exists u' \in \Omega,\  (u' \preceq u) \wedge (\gamma(u') \in A_v). \label{eq:++}
	\end{equation}
	
	Set $\Omega'=\{(u,v) \in \Omega \times \Sigma\;|\; \gamma(u) \in A_v\}$.
	The sentence (\ref{eq:++}) implies that, for any $v \in \Sigma$, there exists $u \in \Omega$ such that $(u,v) \in \Omega'$.
	We next define the preorder $\preceq'$ on $\Omega'$.
	We say that $(u',v') \preceq' (u,v)$ if and only if $u' \preceq u$ and $v' \preceq_{\mathcal A} v$.
	We want to show that $(\Omega',\preceq')$ is a definable directed set.
	The definability of $(\Omega',\preceq')$ is obvious.
	Fix two points $(u,v), (u',v') \in \Omega'$.
	We can take $w \in \Omega$ with $w \preceq u$ and $w \preceq u'$.
	We can also take $v'' \in \Sigma$ with $v'' \preceq_{\mathcal A} v$ and $v'' \preceq_{\mathcal A} v'$.
	We can take $u'' \in \Omega$ such that $u''\preceq w$ and $\gamma(u'') \in A_{v''}$ by the sentence (\ref{eq:++}).
	It implies that $(u'',v'') \in \Omega'$.
	The conditions $(u'',v'') \preceq' (u,v)$ and $(u'',v'') \preceq' (u',v')$ are also satisfied.
	We have demonstrated that $(\Omega',\preceq')$ is a definable directed set.
	
	Let $\pi:\Omega' \to \Omega$ be the natural projection.
	We show that it is downward cofinal.
	 Fix an arbitrary $u \in \Omega$.
	 We can take $(u',v') \in \Omega'$ with $u' \preceq u$ by the sentence (\ref{eq:++}).
	 We have $u''=\pi(u'',v'') \preceq u' \preceq u$ whenever $(u'',v'') \preceq' (u',v')$.
	 It means that $\pi$ is downward cofinal.
	 Consider the subset $\gamma'=\gamma \circ \pi$ of $\gamma$.
	 We demonstrate that $\gamma'$ converges to $x$.
	 In fact, for any $v \in \Sigma$, there exists $u \in \Omega$ such that $\gamma(u) \in A_v$ by the sentence (\ref{eq:++}).
	 Whenever $(u',v') \preceq' (u,v)$, we have $\gamma'(u',v') = \gamma(u') \in A_{v'} \subseteq A_v$.
	 It implies that $\gamma'$ converges to the point $x$.
	 
	The final task is to demonstrate $(1) \Rightarrow (3)$.
	Fix a definable filtered collection of nonempty closed sets $\{C_u\;|\; u \in \Omega\}$.
	Let $(\Omega, \preceq_{\mathcal C})$ be the definable directed set defined by $u \preceq_{\mathcal C} v \Leftrightarrow C_u \subseteq C_v$.
	By Theorem \ref{thm:downward}, there exists a definable downward cofinal map $\gamma:(M^{>0} \times M^{>0},\unlhd) \to (\Omega,\preceq_{\mathcal C})$.
	Lemma \ref{lem:definable_choice} asserts that there exists a definable map $\mu:M^{>0} \times M^{>0} \to X$ such that $\mu(s,t) \in C_{\gamma(s,t)}$ for any $(s,t) \in M^{>0} \times M^{>0}$.
	 Let $\mu_s:M^{>0} \to X$ be the curve defined by $\mu_s(t)=\mu(s,t)$ for any $s >0$.
	 By the condition (1), we have $\convr(\gamma_s) \neq \emptyset$.
	 Using Lemma \ref{lem:definable_choice} once again, we can take a definable map $\kappa:M^{>0} \to X$ such that $\kappa(s) \in \convr(\gamma_s)$ for any $s>0$.
	 Since $\kappa$ is also a curve, there exists $x^* \in \convl(\kappa)$ by the condition (1).
	 
	 We want to show that $x^* \in \bigcap_{u \in \Omega}C_u$.
	 Fix an arbitrary point $u \in \Omega$.
	 Since $\gamma$ is downward cofinal, there exists $(s_u,t_u) \in M^{>0} \times M^{>0}$ such that $\mu(s,t) \in C_u$ whenever $(s,t) \unlhd (s_u,t_u)$.
	 Fix $s \in M^{>0}$ with $0<s<s_u$ for a while.
	 By the definition of $\kappa$, for any $t>t_u$ and any definable neighborhood $A$ of $\kappa(s)$, we have $\mu_s((t,\infty)) \cap A \neq \emptyset$.
	 In particular, we get $C_u \cap A \neq \emptyset$.
	 Since $C_u$ is closed, we have $\kappa(s) \in C_u$.
	 Similarly to the previous argument, we have $\kappa((0,s)) \cap A \neq \emptyset$ for any $0<s<s_u$ and any definable neighborhood $A$ of $x^*$.
	 The intersection $C_u \cap A$ is not empty, and we get $x^* \in C_u$ because $C_u$ is closed.
	 We have demonstrated that $x^* \in \bigcap_{u \in \Omega}C_u$. 
\end{proof}

We employ the definition of definable compactness in \cite{J}. 
\begin{definition}
	Consider an expansion of a dense linear order and a definable topological space $(X,\tau)$.
	It is \textit{definably compact} if the condition (3) of Theorem \ref{thm:compact} is satisfied.
	A definable subset of $X$ is \textit{definably compact} if the set equipped with the relative topology is a definably compact definable topological space.
	The definable topological space $(X,\tau)$ is \textit{locally definably compact} if each point in $X$ has a definably compact neighborhood. 
	
	A definable continuous map $f:(X,\tau_X) \to (Y, \tau_Y)$ is \textit{definably proper} if $f^{-1}(C)$ is definably compact for any definably compact definable subset $C$ of $Y$.
	It is called \textit{definably closed} if the image $f(C)$ is closed for every definable closed subset $C$ of $X$.
\end{definition}

\begin{proposition}\label{prop:affine_case}
	Let $\mathcal M=(M,<,\ldots)$ be a definably complete expansion of a dense linear order.
	Let $X$ be a definable subset of $M^n$ and $\tau_{\text{aff}}$ be the affine topology on $X$.
	The definable topological space $(X,\tau_{\text{aff}})$ is definably compact if and only if $X$ is bounded in $M^n$ and closed in $(M^n,\tau_{M^n})$, where $\tau_{M^n}$ is the product topology of the order topology on $M$.
\end{proposition}
\begin{proof}
	See \cite[Remark 5.6]{FKK}.
\end{proof}

We give an example of a definably compact definable topological space whose underlying set is unbounded.
\begin{example}
	Let us consider a definably complete locally o-minimal structure  $\mathcal M=(M,<,\ldots)$ who has a definable subset $D$ of $M$ such that $c=\inf D \in M$, $\sup D = \infty$ and it is discrete and closed in $M$  in the order topology induced from the order $<$.
	We define a definable topology $\tau_D$ on $D$ as follows:
	Each point other than $c$ is discrete in $\tau_D$ and $\{\{c\} \cup \{x \in D\;|\; x>d\}\;|\;d \in M\}$ is a definable basis of open neighborhoods of the point $c$.
	The definable topological space $(D,\tau_D)$ is definably compact.
	
	In fact, let $\gamma:(a,b) \to D$ be a definable curve with $a \in M \cup\{-\infty\}$ and $b \in M \cup \{\infty\}$.
	We show that $\convr(\gamma)$ is not empty.
	We first consider the case in which the image $\gamma((d,b))$ is bounded for some $a<d<b$.
	Fix $a<d<b$ such that $\gamma((d,b))$ is bounded.
	Take a definable bounded subset $C$ of $D$ with $\gamma((d,b)) \subseteq C$.
	Note that $C$ is closed in $M$ in the order topology by \cite[Proposition 2.8(1)]{FKK}.
	In addition, the relative topology of the topology $\tau_D$ on $C$ coincides with the affine topology on $C$.
	The set $\convr(\gamma)$ is not empty by Proposition \ref{prop:affine_case}. 
	The remaining case is the case in which the image $\gamma((d,b))$ is unbounded for any $a<d<b$.
	It is obvious that $\gamma((d,b))$ has a nonempty intersection with any definable neighborhood of the point $c$ for any $a<d<b$ by the definition of $\tau_D$.
	It implies that $c \in \convr(\gamma)$ and $\convr(\gamma)$ is not empty.
	We can prove that $\convl(\gamma)$ is not empty in the same manner.
	We have proven that $(D,\tau_D)$ is definably compact by Theorem \ref{thm:compact}.
\end{example}

We give three properties of definable compactness which are easily proven.
\begin{lemma}\label{lem:compact_image}
	Consider a definably complete locally o-minimal expansion of an ordered group.
	Let $(X,\tau_X)$ and $(Y,\tau_Y)$ be definable topological spaces and $f : (X,\tau_X) \to (Y,\tau_Y)$ be a definable continuous map.
	If $(X,\tau_X)$ is definably compact, its image $f(X)$ is also definably compact.
\end{lemma}
\begin{proof}
	The proof is straightforward.
	We omit the proof.
\end{proof}

\begin{lemma}\label{lem:proper}
	Consider a definably complete locally o-minimal expansion of an ordered group.
	Let $f:(X,\tau_X) \to (Y,\tau_Y)$ be a definably closed definable continuous map between two definable topological spaces.
	Assume that the fiber $f^{-1}(y)$ is definably compact for each $y \in Y$.
	Then $f$ is definably proper.
\end{lemma}
\begin{proof}
	Let $C$ be a definably compact definable subset of $Y$.
	Let $\{C_u \;|\; u \in \Omega\}$ be a definable filtered collection of nonempty closed subsets of $f^{-1}(C)$.
	We have only to demonstrate that $\bigcap_{u \in \Omega} C_u$ is not an empty set.
	Since $f$ is definably closed, the family $\{f(C_u)\;|\; u \in \Omega\}$ is a definable filtered collection of nonempty closed subsets of $C$.
	We can take $y \in \bigcap _{u \in \Omega} f(C_u)$ because $C$ is definably compact.
	In particular, the intersection $f^{-1}(y) \cap C_u$ is not empty for each $u \in \Omega$.
	Therefore, $\{f^{-1}(y) \cap C_u \;|\; u \in \Omega\}$ is a definable filtered collection of nonempty closed subsets of $f^{-1}(y)$.
	The intersection $f^{-1}(y) \cap \bigcap_{u \in \Omega} C_u$ is not empty because $f^{-1}(y)$ is definably compact.
	It implies that $\bigcap_{u \in \Omega} C_u$ is not an empty set.
\end{proof}

\begin{lemma}\label{lem:compact_closed}
	Consider a definably complete locally o-minimal expansion of an ordered group.
	Let $(X,\tau_X)$ be a definably compact definable space and $C$ be a definable subset.
	If $C$ is closed, $C$ is definably compact.
	The converse holds true if the topological space $(X,\tau)$ is Hausdorff.
\end{lemma}
\begin{proof}
	The proof of the forward implication is easy and omitted.
	We consider the opposite implication.
	Fix a point $x$ in the closure of $C$.
	Consider a definable basis $\mathcal B_x=\{A_u\;|\; u \in \Omega_x\}$ of closed neighborhood of $x$ in $X$.
	Since $\mathcal B_x$ is a definable filtered collection of nonempty closed sets in $X$, $\mathcal C_x=\{A_u \cap C\;|\; u \in \Omega_x\}$ is also a definable filtered collection of nonempty closed sets in $C$.
	Since $C$ is definably compact, we have $\bigcap_{u \in \Omega_x} A_u \cap C \neq \emptyset$.
	On the other hand, we have $\bigcap_{u \in \Omega_x} A_u=\{x\}$ because $(X,\tau)$ is Hausdorff.
	In particular, we obtain $x \in \bigcap_{u \in \Omega_x} A_u \cap C \subseteq C$.
	It implies that $C$ is closed.
\end{proof}

\begin{remark}
	The assumption that $(X,\tau)$ is Hausdorff cannot be dropped in the converse implication of Lemma \ref{lem:compact_closed}.
	We give an example below.
	Consider a definably complete locally o-minimal expansion of an ordered group $\mathcal M=(M,<,+,0,\ldots)$.
	Take a positive element $c \in M$.
	Set $C:=[-c,c]$ and $p=2c \in M$.
	Let $\tau_C$ be the affine topology on $C$ and $\{B_\lambda \subseteq I\;|\; \lambda \in \Lambda\}$ be a definable open base of $\tau_C$.
	We define a definable topology on $X:=C \cup \{p\}$.
	We set $$B'_\lambda = \left\{\begin{array}{ll} B_\lambda \cup \{p\} & \text{ if } 0 \in B_{\lambda},\\
		B_\lambda & \text{ otherwise }
	\end{array}
  \right.
  $$
  for all $\lambda \in \Lambda$.
  The definable family $\{B'_\lambda \subseteq J\;|\; \lambda \in \Lambda\}$ defines a definable topology $\tau_C$ on $C$.
  Two points $0$ and $p$ are not separated in the topology $\tau_X$, and the definable topological space $(X,\tau_X)$ is not Hausdorff.
  The definable subset $C$ is not closed in the topological space $(X,\tau_X)$.
  On the other hand, the relative topology of $\tau_X$ on $C$ coincides with the affine topology $\tau_C$. 
  Therefore, $C$ is a definably compact subset of $X$ by Proposition \ref{prop:affine_case}.
\end{remark}

We obtain a weaker version of a curve selection lemma.
\begin{theorem}[Weak curve selection]\label{thm:curve_sel}
	Consider a definably complete locally o-minimal expansion of an ordered group.
	Let $(X,\tau)$ be a regular locally definably compact definable topological space and $C$ be a definable subset of $X$  which is closed in $(X,\tau)$.
	Let $x \in C$ which is not discrete in $C$.
	There exists a definable curve $\gamma:(0,\infty) \to C \setminus \{x\}$ such that $x \in \convl(\gamma) \cup \convr(\gamma)$.
\end{theorem}
\begin{proof}
	Let $\mathcal M=(M,<,+,0,\ldots)$ be the given definably complete locally o-minimal expansion of an ordered group.
	Take a definably compact definable neighborhood $X'$ of $x$.
	We may assume that $(X,\tau_X)$ is definably compact by considering $X'$ and $X' \cap C$ instead of $X$ and $C$, respectively.
	
	Let $\mathcal B_x=\{A_u\;|\; u \in \Omega_x\}$ be a definable basis of neighborhoods of the point $x$.
	There exists a definable basis of closed neighborhoods of $x$ of the form $\{A_{(s,t)}\;|\;s,t>0\} $ such that $A_{(s',t')} \subseteq A_{(s,t)}$ whenever $(s',t') \unlhd (s,t)$ by Theorem \ref{thm:countability}(3).
	
	Applying Lemma \ref{lem:definable_choice}, we can construct a definable map $\eta:M^{>0} \times M^{>0} \to C \setminus \{x\}$ such that $\eta(s,t) \in A_{(s,t)} \cap C$.
	For any $s>0$, the curve $\gamma_s:(0,\infty) \to C \setminus \{x\}$ is defined by $\gamma_s(t)=\eta(s,t)$.
	Note that the set $\convr(\gamma_s)$ is not empty for each $s>0$ by Theorem \ref{thm:compact} because $(X,\tau)$ is definably compact.
	We consider two separate cases.
	The first case is the case in which $x \in \convr(\gamma_s)$ for some $s>0$.
	We have only to set $\gamma=\gamma_s$ in this case.
	
	The remaining case is the case in which $x \not\in \convr(\gamma_s)$ for all $s>0$.
	Since $C$ is closed, we have $\convr(\gamma_s) \subseteq C \setminus \{x\}$.
	In addition, for any $s>0$ and $t>0$, we get $\gamma_s((t,\infty)) \subseteq \bigcup_{t'>t}A_{(s,t')} \subseteq A_{(s,t)}$ because $A_{(s,t')} \subseteq A_{(s,t)}$ whenever $t'>t$.
	We obtain $\convr(\gamma_s) \subseteq \mycl(\gamma_s((t,\infty))) \subseteq A_{(s,t)}$ because $A_{(s,t)}$ is closed.
	We apply Lemma \ref{lem:definable_choice} once again.
	There exists a definable curve $\gamma:(0,\infty) \to C \setminus \{x\}$ such that $\gamma(s) \in \convr(\gamma_s)$.
	In particular, we have $\gamma(s) \in A_{(s,t)}$ for all $s>0$ and $t>0$.
	We want to show that $x \in \convl(\gamma)$.
	By the definition of $\convl(\gamma)$, we have only to demonstrate $\gamma((0,s)) \cap A_u \neq \emptyset$ for any $s>0$ and any definable neighborhood $A_u \in \mathcal B_x$ of $x$.
	Fix $s>0$ and $A_u \in \mathcal B_x$.
	We can take $(s',t) \in M^{>0} \times M^{>0}$ such that $A_{(s',t)} \subseteq A_u$.
	Set $s''=\min\{s/2,s'\}$.
	We have $\gamma(s'') \in A_{(s'',t)} \subseteq A_{(s',t)}\subseteq A_u$.
	In particular, the intersection $\gamma((0,s)) \cap A_u$ is not empty because it contains the point $\gamma(s'')$.
\end{proof}

Theorem \ref{thm:curve_sel} asserts that we can choose a definable curve when the definable set $C$ is closed.
We consider the case in which we can choose a definable curve even when $C$ is not closed.
\begin{definition}
	Consider an expansion of a dense linear order without endpoints $\mathcal M =(M,<,\ldots)$.
	A definable topological space $(X,\tau)$ has \textit{definable curve selection property} if, for any definable subset $C$ and any point $x \in \partial_{\tau}C$, there exists a definable curve $\gamma:(a,b) \to C$ such that  $a \in M$ and $x \in \convl(\gamma)$.
\end{definition}

We give two sufficient conditions for a definable topological space  to have definable curve selection property.

\begin{proposition}\label{prop:curve_selection0}
	Consider a definably complete locally o-minimal expansion of an ordered group.
	Any definable set $X$ together with the affine topology has definable curve selection property.
\end{proposition}
\begin{proof}
	It is \cite[Corollary 2.9]{Fuji7}.
\end{proof}

\begin{proposition}\label{prop:curve_selection}
	Consider a definably complete locally o-minimal expansion of an ordered field.
	A definable topological space has definable curve selection property.
\end{proposition}
\begin{proof}
	Let $\mathcal M=(M,<,+,\cdot,0,1,\ldots)$ be a definably complete locally o-minimal expansion of an ordered field.
	Let $(X,\tau)$  be a definable topological space and $C$ be a definable subset of $X$.
	Take  $x \in \partial_{\tau}C$. 
	Let $\mathcal B_x=\{A_u\;|\; u \in \Omega_x\}$ be a definable basis of neighborhoods of the point $x$.
	There exists a definable basis of neighborhoods of $x$ of the form $\{A_{(s,t)}\;|\;s,t>0\} $ such that $A_{(s',t')} \subseteq A_{(s,t)}$ whenever $(s',t') \unlhd (s,t)$ by Theorem \ref{thm:countability}(2).
	The definable family $\mathcal B'_x = \{A'_t=A_{t,1/t}\;|\; t>0\}$ is also a definable basis of neighborhoods of $X$.

	We have $A'_t \cap C \neq \emptyset$ because $x \in \partial_{\tau}C$ and $\mathcal B'_x = \{A'_t\;|\; t>0\}$ is a definable basis of neighborhoods of $x$.
	Applying Lemma \ref{lem:definable_choice}, we can construct a definable map $\gamma:M^{>0}  \to C$ such that $\gamma(t) \in A'_t \cap C$.
	It is obvious that $x \in \convl(\gamma)$.
\end{proof}

We can generalize Proposition \ref{prop:curve_selection0} to the case in which the definable topology is metrizable.
We first define definable metric spaces.

\begin{definition}
	Consider an expansion of an ordered group $\mathcal M=(M,<,+,0,\ldots)$.
	A \textit{definable metric space} $(X,d_X)$ is the pair of a definable set $X$ and a definable function $d_X:X  \times X\to M^{ \geq0}:=\{a \in M\;|\; a \geq 0\}$ which satisfies the conditions:
	\begin{enumerate}
		\item[(1)] $d_X(x,y)=0$ if and only if $x=y$ for any $x,y \in X$;
		\item[(2)] $d_X(x,y)=d_X(y,x)$ for any $x,y \in X$;
		\item[(3)] $d_X(x,y)+d_X(y,z) \geq d_X(x,z)$ for any $x,y,z \in X$.
	\end{enumerate}
	Set $B_X(x,\varepsilon):=\{y \in X\;|\; d_X(x,y)<\varepsilon\}$ for $x \in X$ and $\varepsilon>0$.
	The family $\{B_X(x,\varepsilon)\;|\; x \in X, \varepsilon>0\}$ is a definable open base of a definable topology $\tau_X$ on $X$.
	The above definable topology $\tau_X$ induced from the definable distance function $d_X$ is considered when a definable metric space $(X,d_X)$ is given.
	
	A definable curve $\gamma:(0,\varepsilon) \to X$ is \textit{completable in} $X$ if there exists a point $p \in X$ such that, for any $0<\varepsilon'<\varepsilon$ and any definable neighborhood $U$ of $p$ in $X$, the intersection $\gamma((0,\varepsilon')) \cap U$ is not empty.
	Here, we assume that $\varepsilon \in M^{>0} \cup \{+\infty\}$.
	We write $\gamma \to p$ in this case.
	Note that the domain of definition of the curve $\gamma$ is the open interval $(0,\varepsilon)$ and its left endpoint is always zero here differently from Definition \ref{def:limit}. 
\end{definition}

The following proposition is a generalized version of Proposition \ref{prop:curve_selection0}.
\begin{proposition}\label{prop:curve_selection2}
	Consider a definably complete locally o-minimal expansion of an ordered group.
	A definable metric space has definable curve selection property.
\end{proposition}
\begin{proof}
	Let $(X,d_x)$ be a definable metric space and $C$ be a definable subset of $X$. 
	Let $x \in \partial_\tau(C)$.
	There exists a definable curve $\gamma:(0,\delta) \to X$ such that $\gamma(t) \in C \cap B_X(x,t)$ by Lemma \ref{lem:definable_choice}.
	It is obvious that $x \in \convl(\gamma)$.
\end{proof}

We give several technical lemmas on definable metric spaces used in the next section.
\begin{lemma}\label{lem:unique_lim}
	Consider a definably complete locally o-minimal expansion of an ordered group $\mathcal M=(M,<,+,0,\ldots)$.
	Let $(X,d_X)$ be a definable metric space and $\gamma:(0,\varepsilon) \to X$ be a definable curve.
	The set $\convl(\gamma)$ is either empty or a singleton. 
	Furthermore, if $\convl(\gamma)$ is a singleton $\{p\}$, we have $\gamma \to p$.
\end{lemma}
\begin{proof}
	Let $x, y \in \convl(\gamma)$.
	We have only to prove that $x=y$.
	Assume for contradiction that $x \neq y$.
	Set $d=d_X(x,y)>0$.
	We have $B_X(x,d/2) \cap B_X(y,d/2)=\emptyset$ by the triangle inequality.
	Consider the definable set  $S:=\{0<t<\varepsilon\;|\; \gamma(t) \in B_X(x,d/2)\}$.
	Since $\mathcal M$ is locally o-minimal, there exists $\delta>0$ such that either $(0,\delta) \subseteq S$ or $(0,\delta) \cap S = \emptyset$.
	We have $(0,\delta) \subseteq S$ because $x \in \convl(\gamma)$.
	Since $B_X(x,d/2) \cap B_X(y,d/2)=\emptyset$, we have $\gamma((0,\delta))  \cap B_X(y,d/2)=\emptyset$.
	It contradicts the assumption that $y \in \convl(\gamma)$.
	The `furthermore' part is obvious from the definitions of $\convl(\gamma)$  and $\gamma \to p$.
\end{proof}

\begin{lemma}\label{lem:cont_equiv}
	Consider a definably complete locally o-minimal expansion of an ordered group $\mathcal M=(M,<,+,0,\ldots)$.
	Let $(X,d_X)$ and $(Y,d_Y)$ be definable metric spaces.
	A definable map $f:X \to Y$ is continuous if and only if, for any point $p \in X$ and any definable curve $\gamma$ completable in $X$ with $\gamma \to p$, the curve $f \circ \gamma$ is also completable in $Y$ and $f \circ \gamma \to f(p)$.
\end{lemma}
\begin{proof}
	We first show the `only if' part.
	Consider a positive $\varepsilon>0$, $p \in X$  and a definable map $\gamma:(0,a) \to X$ with $\gamma \to p$.
	We have only to demonstrate that there exists $\delta>0$ such that $f(\gamma(0,\delta)) \subseteq B_Y(f(p),\varepsilon)$.
	Since $f$ is continuous, we can choose $\delta'>0$ so that $f(B_X(p,\delta')) \subseteq B_Y(f(p),\varepsilon)$.
	By local o-minimality, if we choose a sufficiently small $\delta>0$, we have either $\gamma((0,\delta)) \subseteq B_X(p,\delta')$ or $\gamma((0,\delta)) \cap  B_X(p,\delta')=\emptyset$.
	The former inclusion $\gamma((0,\delta)) \subseteq B_X(p,\delta')$ holds true because $\gamma \to p$.
	We finally get the inclusion $f(\gamma(0,\delta)) \subseteq f(B_X(p,\delta')) \subseteq  B_Y(f(p),\varepsilon)$.
	
	We next show the `if' part.
	We demonstrate the contraposition.
	Assume that $f$ is not continuous.
	There exist $p \in X$ and $\varepsilon>0$ such that, for any $\delta>0$, we can take $x \in B_X(p,\delta)$ with $f(x) \not\in B_Y(f(p),\varepsilon)$. 
	By Lemma \ref{lem:definable_choice}, we can take $\gamma:(0,\infty) \to X \setminus \{p\}$ such that $\gamma(t) \in B_X(p,t)$ and $f(\gamma(t)) \not\in B_Y(f(p), \varepsilon)$.
	It implies that $\gamma \to p$ and $f \circ \gamma \not\to f(p)$. 
\end{proof}


We also get the following theorem by applying Example \ref{ex:awk}.
\begin{theorem}[Definable open cover theorem]\label{thm:bounded_para}
	Let $\mathcal M=(M,<,+,0,\ldots)$ be a definably complete locally o-minimal expansion of an ordered group.
	Let $(X,\tau)$ be a definably compact definable topological space and $\{U_t|t \in P\}$ be a definable family of definable open subsets of $X$ such that $X=\bigcup_{t \in P}U_t$.
	There exists a definable bounded subset $Q$ of $P$ such that $X=\bigcup_{t \in Q}U_t$.
\end{theorem}
\begin{proof}
	Consider the awkward linear order $\awk$ on $P$ defined in Example \ref{ex:awk}.
	The pair $(P,\awk)$ is a definable directed set by Lemma \ref{lem:obv1}.
	Set $C_t= \bigcap_{s \awk t} (X \setminus U_s)$ for all $t \in P$.
	We want to show that $C_t$ is empty for some $t \in P$.
	Assume for contradiction that $C_t$ is not empty for any $t \in P$.
	The family $\{C_t\;|\; t \in P\}$ is a definable filtered collection of nonempty closed subsets of $X$.
	It has a nonempty intersection by Theorem \ref{thm:compact}.
	It means that $\{U_t|t \in P\}$ is not an open cover of $X$.
	Contradiction.
	
	Take $v \in P$ so that $C_v=\emptyset$.
	Set $Q=\{t \in P\;|\; t \awk v\}$.
	It is a bounded definable set by the definition of the awkward linear order $\awk$.
\end{proof}

\begin{remark}
	Consider the case in which $\mathcal M=(M,<,+,0,\ldots)$ is an almost o-minimal expansion of an ordered group in Theorem \ref{thm:bounded_para}.
	See \cite{Fuji6} for the definition of almost o-minimal structures.
	Note that $\mathcal M$ is a definably complete locally o-minimal structure by \cite[Lemma 4.6]{Fuji6}.
	Let  $\myindr(\mathcal M)$ be the o-minimal structure defined in the last part of \cite[Section 4.3.1]{Fuji6}.
	Any bounded set definable in $\mathcal M$ is definable in $\myindr(\mathcal M)$.
	Theorem \ref{thm:bounded_para} says that, if a closed bounded definable set $X$ which has an open cover definable in $\mathcal M$, it has an open cover definable in $\myindr(\mathcal M)$ which is a subfamily of the given open cover.
\end{remark}

\section{Definably compact definable group}\label{sec:group}

As an application of previous sections, we study definable topological groups.
\begin{definition}
	Consider an expansion of an ordered group.
	A \textit{definable group} is a group $(G,e,\cdot)$ such that the set $G$ and the map $\cdot: (a,b) \mapsto a \cdot b$ are definable.
	We sometimes simply denote it by $G$.
	A definable group is \textit{definably simple} if it does not have a nontrivial definable normal subgroups.
	A \textit{definable topological group} $(G,e,\cdot,\tau)$ is a definable group equipped with a definable topology $\tau$ whose multiplication and inverse are continuous.   
	We also simply denote it by $(G,\tau)$.
\end{definition}

We first recall Pillay and Wencel's result.
\begin{remark}
	A definably complete locally o-minimal structure is a first-order topological structure in the sense of \cite{Pillay}.
	Consider a definably complete locally o-minimal expansion of an ordered group.
	A definable group has a definable topology with which the definable group is a definable topological group by \cite[Proposition 2.8(7)]{FKK} and \cite[Theorem 3.5]{Wencel}.
\end{remark}

The following lemma is not used in this paper, but it is worth mentioned.
\begin{lemma}\label{lem:group_basic1}
	Consider a definably complete locally o-minimal expansion of an ordered group $\mathcal M=(M,<,+,0,\ldots)$.
	Let $(G,\tau)$ be a definable topological group having definable curve selection property.
	Let $U$ be a definable neighborhood of the identity element $e$ in $G$ and $C$ be a definably compact definable subset of $G$.
	The intersection $U'=\bigcap_{h \in C} hUh^{-1}$ is also a neighborhood of $e$. 
\end{lemma}
\begin{proof}
	Assume for contradiction that $U'$ is not a neighborhood of $e$.
	There exists a definable curve $\gamma:(a,b) \to U \setminus U'$ such that $a \in M$ and $e \in \convl(\gamma)$ by the definable curve selection property.
	Set $C_t=\{h \in C\;|\; \gamma(t) \not\in hUh^{-1}\}$ for all $a<t<b$.
	The definable set $C_t$ is not empty by the definition of $U'$.
	By Lemma \ref{lem:definable_choice}, there exists a definable curve $\eta:(a,b) \to C$ such that $\eta(t) \in C_t$ for all $a<t<b$.
	Since $C$ is definably compact, we can take $h_0 \in \convl(\eta)$ by Theorem \ref{thm:compact}.
	Consider the definable continuous map $C \times G \to G$ given by $(h,g) \mapsto h^{-1}gh$.
	There exist definable neighborhoods $V_e$ and $V_{h_0}$ of $e$ and $h_0$ in $G$ and $C$, respectively, such that $h^{-1}gh \in U$ for all $g \in V_e$ and $h \in V_{h_0}$.
	By the definition of the set of left convergence and local o-minimality, there exists $a<s<b$ such that $\gamma((a,s)) \subseteq V_e$ and $\eta((a,s)) \subseteq V_{h_0}$.
	Take $a<t<s$.
	We have $\eta(t)^{-1}\gamma(t)\eta(t) \in U$ and $\gamma(t) \in \eta(t)U\eta(t)^{-1}$.
	It contradicts the condition that $\eta(t) \in C_t$.
\end{proof}

We investigate that a definably simple definable topological group whose topology is definably compact.
We want to show that it is either discrete or definably connected.
A connected component of a topological group which contains the identity element is a normal subgroup.
The standard proof is as follows:
Let $G_0$ be the connected component of a topological group $G$ which contains the identity element $e$.
The restriction $m_0:G_0 \times G_0 \to G$ of the multiplication in $G$ to $G_0 \times G_0$ is continuous and $G_0 \times G_0$ is connected.
The inverse image $m_0^{-1}(G_0)$ is a nonempty closed and open subset of $G_0 \times G_0$, and we have $m_0(G_0 \times G_0)=G_0$.
It means that the multiplication is closed in $G_0$.
We can prove that the inverse is closed in $G_0$ in the same manner.
Since $gG_0g^{-1}$ is connected and contains the identity element $e$, we also have $gG_0g^{-1}=G_0$ for each $g \in G$.
It is trivial that a simple topological group is either discrete or connected.

The above classical proof is still valid for topological group definable in an o-minimal structure using definably connected components instead of connected components.
On the other hand, for definably complete locally o-minimal structure, the notion of `definably connected component' is currently unavailable.
The author does not know whether a connected component of a set definable in a locally o-minimal expansion of the set of reals $\mathbb R$ is definable or not.
We found an alternative proof which is only applicable when the definable topological group is definably compact.
We prove that a definably simple definable topological group such that it is regular and definably compact as a definable topological space is either discrete or definably connected.
A definable set is \textit{definably connected} in a definable topology if it has no nontrivial open and closed definable subset in the given topology.

\begin{theorem}\label{thm:simple_group}
	Consider a definably complete locally o-minimal expansion of an ordered group.
	A definably simple definable topological group $(G,\tau)$ which is regular, Hausdorff and definably compact as a definable topological space is either discrete or definably connected in the topology $\tau$.
\end{theorem}
\begin{proof}
	Let $\mathcal M=(M,<,+,0,\ldots)$ be the given definably complete locally o-minimal expansion of an ordered group.
	We only consider the topology $\tau$ in this proof.
	
	We first consider the case in which the identity element $e$ is discrete in $G$.
	The multiplication by an arbitrary fixed element induces a definable homeomorphism.
	If $e$ is a discrete point, then any element $g \in G$ is also a discrete point.
	Therefore, $(G,\tau)$ is discrete in this case.
	
	We next consider the case in which $e$ is not a discrete point.
	Assume for contradiction that $G$ is not definably connected.
	There exist nonempty definable closed and open subsets $G_1$ and $G_2$ of $G$ such that $G=G_1 \cup G_2$ and $G_1 \cap G_2 = \emptyset$.
	We may assume that $e \in G_1$ without loss of generality.
	Set $H=\{g \in G\;|\; gG_1=G_1\}$, which is the stabilizer of $G_1$.
	It is trivially a definable subgroup of $G$.
	It is also trivial that $H \subseteq G_1$ because the identity element $e$ belongs to $G_1$.
	In particular, we have $H \neq G$.
	We show that $H$ is open.
	Set $H_1=\{g \in G\;|\; gG_1 \subseteq G_1\}$ and $H_2=\{g \in G\;|\; g^{-1}G_1 \subseteq G_1\}$.
	We obviously have $H=H_1 \cap H_2$.
	Therefore, we have only to prove both $H_1$ and $H_2$ are open.
	We only prove that $H_1$ is open.
	We can prove that $H_2$ is open similarly.
	Since $G_1$ is closed in $G$, it is definably compact by Lemma \ref{lem:compact_closed}.
	Note that $G \times G_1$ is definably compact by Lemma \ref{lem:conv_formula}(1) and Theorem \ref{thm:compact}. 
	Consider the definable maps $\rho:G \times G_1 \to G$ and $\pi:G \times G_1 \to G$ defined by $\rho(g,h)=gh$ and $\pi(g,h)=g$, respectively.
	It is obvious that $H_1=G \setminus \pi(\rho^{-1}(G_2))$.
	The inverse image $\rho^{-1}(G_2)$ is closed because $G_2$ is closed and $\rho$ is continuous.
	Therefore, it is definably compact by Lemma \ref{lem:compact_closed}.
	The image $\pi(\rho^{-1}(G_2))$ is definably compact by Lemma \ref{lem:compact_image} and it is closed in $G$ by Lemma \ref{lem:compact_closed}.
	We have demonstrated that $H_1$ is open.
	We have also demonstrated that $H$ is open.
	
	We next consider the intersection $N=\bigcap_{g \in G}gHg^{-1}$.
	A standard group theoretic argument implies that $N$ is a normal subgroup of $G$ because $H$ is a subgroup of $G$.
	We omit the details.
	It is also trivial that $N$ is definable.
	Since $H \neq G$, we also have $N \neq G$.
	We get $N=\{e\}$ because $G$ is definably simple.
	Since $\{e\}$ is not a discrete point by the assumption, there exists a definable curve $\gamma:(0,\infty) \to G \setminus \{e\}$ such that $e \in \convl(\gamma) \cup \convr(\gamma)$ by Theorem \ref{thm:curve_sel}.
	We only consider the case in which $e \in \convl(\gamma)$.
	The other case is similar.
	The equality $N=\{e\}$ implies that, for any $t>0$, there exists $g \in G$ such that $\gamma(t) \not\in gHg^{-1}$.
	By Lemma \ref{lem:definable_choice}, there exists a definable curve $\iota:M^{>0} \to G$ such that $\gamma(t) \not\in \iota(t)H\iota(t)^{-1}$.
	Since $G$ is definably compact, the set $\convl(\iota)$ is not empty.
	We take an element $g_0 \in \convl(\iota)$.
	Consider the map $\eta:G \times G \to G$ defined by $\eta(g,h)=h^{-1}gh$ and the definable curve $\gamma':M^{>0} \to G$ defined by $\gamma'(t)=\eta(\gamma(t),\iota(t))=\iota(t)^{-1}\gamma(t)\iota(t)$.
	We have $e=g_0^{-1}eg_0 \in \convl(\gamma')$ by Lemma \ref{lem:conv_formula}(1),(2) because $e \in \convl(\gamma)$ and $g_0 \in \convl(\iota)$.
	On the other hand, $\gamma'(t)$ is not contained in $H$ for each $t>0$.
	Since $H$ is open, we have $e \in \convl(\gamma') \subseteq G \setminus H$.
	It is a contradiction.
\end{proof}

We next study definable quotient groups.

\begin{definition}
%

Let $(G,\tau_G)$ be a definable topological group and $(X,\tau_X)$ be a definable topological space.
Consider a definable continuous left $G$-action $\mathfrak a:G \times X \to X$.
The pair of a definable topological space $(Q,\tau_Q)$ and a definable continuous map $\pi:(X,\tau_X) \to (Q,\tau_Q)$ is the \textit{definable (left) quotient of $X$ by  $G$} if, for any definable topological space $(T,\tau_{T})$ and a definable continuous map $\varphi:(X,\tau_X) \rightarrow (T,\tau_T)$ satisfying the equality $\varphi(x)=\varphi(gx)$ for any $g \in G$ and $x \in X$, there exists a definable continuous map $\psi:(Q,\tau_Q) \rightarrow (T,\tau_T)$  such that $\varphi=\psi\circ \pi$.
We also call $(Q,\tau_Q)$ the \textit{definable quotient space of $X$ by  $G$}.
The definable quotient space is unique up to definable homeomorphism.
We denote the definable quotient space by $X/G$. 
We call the map $\pi$ the \textit{canonical projection} of $X$ onto $Q$. 
We define the definable quotient space for definable continuous right $G$-action in the same manner.

When $G$ is a definable subgroup of a definable topological group $(G_0,\tau_{G_0})$, the multiplication in $G_0$ induces the natural definable continuous left and right actions on $(G_0,\tau_{G_0})$ by $(G,\tau_G)$, where $\tau_G$ is the relative topology of $\tau_{G_0}$ on $G$.
	When $G$ is normal, $(Q,\tau_Q)$ is a definable topological group by Proposition \ref{prop:exist_quo} below.
	We call $(Q,\tau_Q)$ the \textit{definable quotient group of $G_0$ by  $G$} in this case.
	The topology $\tau_Q$ is called \textit{definable quotient topology}.
\end{definition}

\begin{proposition}\label{prop:exist_quo}
	Consider a definably complete locally o-minimal expansion of an ordered group.
	Let $(G,\tau_G)$ be a definable topological group and $(X,\tau_X)$ be a definable topological space.
	Consider a definable continuous left $G$-action $\mathfrak a:G \times X \to X$.
	Then there exists the definable quotient of $X$ by  $G$.
	Furthermore, if $X$ is a definable group, $G$ is a definable normal subgroup of $X$ and the action $\mathfrak a$ is induced from the multiplication in $X$, the definable quotient space is a definable topological group whose multiplication is defined by $g_1G \cdot g_2G = g_1g_2G$.
\end{proposition}
\begin{proof}
	Since the definable set $E=\{(x_1,x_2) \in X \times X\;|\; \exists g \in G,\ x_2=gx_1 \}$ is a definable equivalence relation, there exists a definable subset $Q$ of $X$ such that $Q$ intersect with the orbit $Gx$ at exactly one point for any $x \in X$ by Lemma \ref{lem:definable_choice}.
	Let $\pi:X \to Q$ be the definable map sending $x \in X$ to the unique intersection point of $Q$ with $Gx$.
	We define a definable topology $\tau_Q$ on $Q$.
	A definable subset $S$ of $Q$ is defined to be open if and only if $\pi^{-1}(S)=\bigcup_{g \in G}gS$ is open.
	Note that $\pi^{-1}(\pi(U))=\bigcup_{g \in G}gU$ for any subset $U$ of $X$.
	It is trivial that the topology $\tau_Q$ is the definable topology having the definable open base $\{\pi(U)=(\bigcup_{g \in G}gU) \cap Q\;|\; U \in \mathcal B\}$, where $\mathcal B$ is a definable open base of the topology $\tau_G$.
	It is straightforward to demonstrate that the pair $((Q,\tau_Q),\pi)$ is the definable quotient. 
	The proof of `furthermore' part is also straightforward.
	We omit their proofs.
\end{proof}
%

Let $(G,\tau_G)$ be a definable topological group and $H$ be a closed definable normal subgroup of $G$.
When $G$ is definably compact, both $H$ and the quotient topological space $G/H$ are definably compact by Lemma \ref{lem:compact_image} and Lemma \ref{lem:compact_closed}.
The following theorem says that its inverse also holds true.
\begin{theorem}\label{thm:quotient}
	Consider a definably complete locally o-minimal expansion of an ordered group.
	Let $(G,\tau_G)$ be a definable topological group and $H$ be a definable normal subgroup of $G$.
	Let $\tau_H$ be the relative topology on $H$.
	The definable topological group $(G,\tau_G)$ is definably compact if both the definable topological subspace $(H,\tau_H)$ and the definable quotient group $(Q:=G/H,\tau_Q)$ are definably compact.
\end{theorem}
\begin{proof}
	Let $\pi:G \to Q$ be the canonical projection of $G$ onto $Q$.
	Let $\gamma:(a,b) \to G$ be a definable curve.
	We have only to demonstrate that both $\convl(\gamma)$ and $\convr(\gamma)$ are nonempty by Theorem \ref{thm:compact}.
	We only show that $\convl(\gamma)$ is not empty.
	The proof for $\convr(\gamma)$ is similar and it is omitted.
	
	Since $Q$ is definably compact, we can take $q \in \convl(\pi \circ \gamma)$ by Theorem \ref{thm:compact}.
	We fix a point $g \in G$ with $\pi(g)=q$.
	Set $H'=gH=Hg$.
	The definable topological space $H'$ is definably compact because $H$ is definably compact and the definable map given by $g_1 \mapsto gg_1$ is a definable homeomorphism.
	
	 By Theorem \ref{thm:countability}(2), there exists a definable basis $\{A_{(s,t)}\;|\;(s,t) \in M^{>0} \times M^{>0}\}$ of open neighborhoods of the identity element $e$ in $G$ satisfying $A_{(s',t')} \subseteq A_{(s,t)}$ whenever $(s',t') \unlhd (s,t)$.
	 It is obvious that $\{g'A_{(s,t)}\;|\;(s,t) \in M^{>0} \times M^{>0}\}$ is a definable basis of open neighborhoods of $g'$ in $G$ for any point $g'$ in $G$.
	  We can easily deduce the equality 
	 \begin{equation}
	 	\pi^{-1}(\pi(gA_{(s,t)}))=\bigcup_{h \in H} hgA_{(s,t)}=\bigcup_{h \in H} ghA_{(s,t)}=\bigcup_{h' \in H'}h'A_{(s,t)} \label{eq:aaa}
	 \end{equation}
	 using the assumption that $H$ is normal.
	 We omit the proof of the above equality.
	The equality (\ref{eq:aaa}) and the definition of definable quotient topology imply that the family  $\{\pi(gA_{(s,t)})\;|\;(s,t) \in M^{>0} \times M^{>0}\}$ is a definable basis of open neighborhoods of $q$ in $Q$.
	 Since $q \in \convl(\pi\circ\gamma)$, we have $\pi(\gamma((a,u))) \cap \pi(gA_{(s,t)})\neq \emptyset$ for any $(s,t) \in M^{>0} \times M^{>0}$ and $u \in (a,b)$.	 
	 Applying the equality (\ref{eq:aaa}), there exists $h \in H'$ such that $$\gamma((a,u)) \cap hA_{(s,t)} \neq \emptyset.$$
	 It implies that $B(s,t,u):= \{h \in H'\;|\; \gamma((a,u)) \cap hA_{(s,t)} \neq \emptyset\}$ is not empty for any $(s,t) \in M^{>0} \times M^{>0}$ and $u \in (a,b)$.
	 
	 We can take a definable map $\eta:M^{>0} \times M^{>0} \times (a,b) \to H'$ so that $\eta(s,t,u) \in B(s,t,u)$ by Lemma \ref{lem:definable_choice}.
	 In particular, we have $$\gamma((a,u)) \cap \eta(s,t,u)A_{(s,t)} \neq \emptyset.$$
	 We define the preorder $\unlhd'$ on $M^{>0} \times M^{>0} \times (a,b)$.
	 We define $(s',t',u') \unlhd' (s,t,u) $ if and only if $(s',t') \unlhd (s,t)$ and $u' \leq u$.
	 The pair $(M^{>0} \times M^{>0} \times (a,b), \unlhd')$ is a directed set.
	 Therefore, we may consider that $\eta$ is a definable net from the directed set $(M^{>0} \times M^{>0} \times (a,b), \unlhd')$ to $H'$.
	 Since $H'$ is definably compact, there exist a definable downward cofinal map $f:(\Omega', \prec') \to (M^{>0} \times M^{>0} \times (a,b), \unlhd')$ and a point $\overline{h} \in H'$ such that $\eta \circ f$ converges to $\overline{h}$ by Theorem \ref{thm:compact}.
	 
	 We want to show that $\overline{h} \in \convl(\gamma)$.
	 We fix an arbitrary definable neighborhood $B$ of $e$ in $G$ and an arbitrary point $u \in (a,b)$.
	 We have only to demonstrate $$\gamma((a,u)) \cap \overline{h}B \neq \emptyset.$$
	 We may assume that $B$ is open by taking a smaller open neighborhood of $e$ contained in $B$ if necessary.
	 Let $m:G \times G \to G$ be the multiplication in $G$ defined by $m(g_1,g_2)=g_1 \cdot g_2$.
	 Since $m$ is continuous and $m(e,e)=e$, the inverse image $m^{-1}(B)$ is an open neighborhood of the point $(e,e)$.
	 We can take definable open neighborhoods $B_1$ and $B_2$ of $e$ in $G$ such that $B_1 \times B_2 \subseteq m^{-1}(B)$ by the definition of product topology.
	 It implies that $B_1 \cdot B_2 \subseteq B$.
	 Since $\{A_{(s,t)}\;|\;(s,t) \in M^{>0} \times M^{>0}\}$ is a definable basis of open neighborhoods of $e$ in $G$, there exists $(s,t) \in M^{>0} \times M^{>0}$ with $A_{(s,t)} \subseteq B_2$.
	 We fix such an $(s,t)$.
	 
	 Since $\eta \circ f$ converges to $\overline{h}$, there exists $x_1 \in \Omega'$ such that $\eta(f(x)) \in \overline{h}B_1$ whenever $x \in \Omega'$ and $x \preceq' x_1$.
	 On the other hand, by the definition of downward cofinal maps, there exists $x_2 \in \Omega'$ such that $f(x) \unlhd' (s,t,u)$ whenever $x \in \Omega'$ and $x \preceq' x_2$.
	 We can take $\widetilde{x} \in \Omega'$ with $\widetilde{x} \preceq' x_1$ and $\widetilde{x} \preceq' x_2$ because $(\Omega',\preceq')$ is a directed set.
	 Put $(\overline{s},\overline{t},\overline{u})=f(\widetilde{x}) \in M^{>0} \times M^{>0} \times (a,b)$.
	 We get $(\overline{s},\overline{t}) \unlhd (s,t)$, $\overline{u} \leq u$ and $\eta(\overline{s},\overline{t},\overline{u}) \in \overline{h}B_1$. 
	 We have
	 \begin{align*}
	 	\eta(\overline{s},\overline{t},\overline{u}) A_{(\overline{s},\overline{t})} \subseteq \overline{h}B_1\cdot A_{(\overline{s},\overline{t})} \subseteq \overline{h}B_1\cdot A_{(s,t)} \subseteq \overline{h}B_1 \cdot B_2 \subseteq  \overline{h}B
	 \end{align*}
     using the inclusions $\eta(\overline{s},\overline{t},\overline{u}) \in \overline{h}B_1$, $A_{(\overline{s},\overline{t})} \subseteq A_{(s,t)} \subseteq B_2$ and $B_1\cdot B_2 \subseteq B$.
     The inclusion $\gamma((a,\overline{u})) \subseteq \gamma((a,u))$ is obvious and we obtain the inclusion$$\gamma((a,\overline{u})) \cap \eta(\overline{s},\overline{t},\overline{u}) A_{(\overline{s},\overline{t})}  \subseteq \gamma((a,u)) \cap \overline{h}B.$$
     It implies that $\gamma((a,u)) \cap \overline{h}B$ is not empty because $\gamma((a,\overline{u})) \cap \eta(\overline{s},\overline{t},\overline{u}) A_{(\overline{s},\overline{t})}$ is not empty by the definition of the definable map $\eta$.
     We have completed the proof.
\end{proof}

We next consider the case in which the definable topology is induced from a definable distance function. 
\begin{definition}
		A \textit{definable metric group} $(G,d_G)$ is a definable metric space such that $(G,\tau_G)$ is a definable topological group, where $\tau_G$ is the topology induced from $d_G$. 
\end{definition}

The assertion (2) of the following proposition claims that Theorem \ref{thm:quotient} holds true im more general setting when $G$ and $X$ are definable metric spaces.
\begin{proposition}\label{prop:inv_dist}
	Consider a definably complete locally o-minimal expansion of an ordered group $\mathcal M=(M,<,+,0,\ldots)$.
	Let $(G,d_G)$ be a definably compact definable metric group and $(X,d_X)$ be a definable metric space.
	Consider a definable continuous left $G$-action on $X$.
	Let $\pi:X \to Q$ be the definable quotient map.
	Then the following assertions hold true.
	\begin{enumerate}
		\item[(1)] The map $\pi$ is a definably closed map.
		\item[(2)] The definable topological space $X$ is definably compact if and only if the quotient space $Q$ is definably compact.
	\end{enumerate}
\end{proposition}
\begin{proof}
	We first demonstrate the assertion (1).
	Let $C$ be a definable closed subset of $X$.
	We have only to demonstrate that $C':=\pi^{-1}(\pi(C))=\bigcup_{g \in G}gC$ is closed by the definition of definable quotient topology.
	Assume for contradiction that $C'$ is not closed.
	Take a point $p \in \partial X$.
	There exists a definable curve $\gamma: (0,\infty) \to C'$ such that $\gamma \to p$ by Proposition \ref{prop:curve_selection2} and Lemma \ref{lem:unique_lim}.
	There exist definable curves $h:(0,\infty) \to G$ and $f:(0,\infty) \to C$ with $\gamma(t)=h(t) \cdot f(t)$ for each $t>0$ by Lemma \ref{lem:definable_choice}.
	Since $G$ is definably compact, we have $h \to g_0$ for some $g_0 \in G$ by Theorem \ref{thm:compact} and Lemma \ref{lem:unique_lim}.
	Since $f(t)=(h(t))^{-1}\cdot \gamma(t)$, we have $f \to g_0^{-1} \cdot p$ by  Lemma \ref{lem:cont_equiv}.
	We have $g_0^{-1} \cdot p \in C$ because $C$ is closed.
	It implies that $p \in g_0C \subseteq C'$, which is a contradiction.
	
	The assertion (2) is easily proven.
	The `only if' part follows from Lemma \ref{lem:compact_image}.
	The fiber of every point in $Q$ under $\pi$ is definably compact.
	The `if' part follows from this fact, the assertion (1) and Lemma \ref{lem:proper}.
\end{proof}

We  prove that definable quotient $X/G$ is a definable metric space when both $G$ and $X$ are definable metric spaces and $G$ is definably compact.

\begin{theorem}\label{thm:inv_dist}
	Consider a definably complete locally o-minimal expansion of an ordered group $\mathcal M=(M,<,+,0,\ldots)$.
	Let $(G,d_G)$ be a definably compact definable metric group and $(X,d_X)$ be a definable metric space.
	Consider a definable continuous left $G$-action on $X$.
	Then the definable quotient space $(Q,\tau_Q)$ of $X$ by $G$ is a definable metric space.
\end{theorem}
\begin{proof}
	We want show that the $G$-invariant definable function $\widetilde{d_X}:X \times X \to M_{\geq 0}$ given by $$\widetilde{d_X}(x,y)=\sup \{d_X(gx,gy)\;|\;g \in G\}$$ is a continuous distance function.
	Let $x,y \in X$ be fixed.
	We first demonstrate the following claim:
	
	\medskip
	\textbf{Claim 1.} Let $x,y \in X$ . There exists $g_{x,y} \in G$ such that $$\widetilde{d_X}(x,y)=\sup \{d_X(gx,gy)\;|\;g \in G\}=d_X(g_{x,y}x,g_{x,y}y).$$
	\begin{proof}[Proof of Claim 1.]
		Consider the function $\rho_{x,y}:G \ni g \mapsto d_X(gx,gy) \in M$.
		It is obviously continuous.
		By Lemma \ref{lem:compact_image}, its image $\rho_{x,y}(G)$ is definably compact.
		Note that a definable set in $M$ is definably compact if and only if it is bounded and closed by Proposition \ref{prop:affine_case}.
		It implies that there exists $g_{x,y} \in G$ such that $\widetilde{d_X}(x,y)=\sup \{d_X(gx,gy)\;|\;g \in G\}=d_X(g_{x,y}x,g_{x,y}y)$.
	\end{proof}
	
	We first demonstrate that $\widetilde{d_X}$ is continuous.
	Let $p_1,p_2$ be arbitrary points in $X$ and $\gamma_i:(0,\varepsilon) \to X$ be definable curves in $X$ with $\gamma_i \to p_i$ for $i=1,2$.
	We have only to demonstrate that $\widetilde{d_X}\circ (\gamma_1,\gamma_2) \to \widetilde{d_X}(p_1,p_2)$ by Lemma \ref{lem:cont_equiv}.
	There exists a definable map $\eta:(0,\varepsilon) \to G$ with $\widetilde{d_X}(\gamma_1(t),\gamma_2(t)) = d_X(\eta(t)\gamma_1(t),\eta(t)\gamma_2(t))$ by Claim 1 and Lemma \ref{lem:definable_choice}.
	Since $(G,\tau_G)$ is definably compact, the set $\convl(\eta)$ is not empty by Theorem \ref{thm:compact}.
	It is a singleton by Lemma \ref{lem:unique_lim}.
	Take the unique point $g \in \convl(\eta)$.
	We have $\eta \to g$ by Lemma \ref{lem:unique_lim}.
	We have only to demonstrate $\widetilde{d_X}(p_1,p_2)=d_X(gp_1,gp_2)$.
	In fact, if this equality holds true, we get $\widetilde{d_X}(\gamma_1(t),\gamma_2(t)) = d_X(\eta(t)\gamma_1(t),\eta(t)\gamma_2(t)) \to d_X(gp_1,gp_2)=\widetilde{d_X}(p_1,p_2)$ as $t \to 0$ by Lemma \ref{lem:cont_equiv}.
	
	Assume for contradiction that $\widetilde{d_X}(p_1,p_2) \neq d_X(gp_1,gp_2)$.
	There exists $h \in G$ such that $\widetilde{d_X}(p_1,p_2)=d_X(hp_1,hp_2)>d_X(gp_1,gp_2)$ by Claim 1.
	Set $\eta'(t) = hg^{-1}\eta(t)$.
	We have $\eta'(t) \to h$.
	We get $ d_X(\eta'(t)\gamma_1(t),\eta'(t)\gamma_2(t)) \to d_X(hp_1,hp_2)$ by Lemma \ref{lem:cont_equiv}.
	Since $d_X(hp_1,hp_2)>d_X(gp_1,gp_2)$, we have $d_X(\eta'(t)\gamma_1(t),\eta'(t)\gamma_2(t))>d_X(\eta(t)\gamma_1(t),\eta(t)\gamma_2(t))$ for sufficiently small $t>0$.
	This inequality contradicts the definition of $\eta$.
	
	We next demonstrate that $\widetilde{d_X}$ is a distance function.
	It is obvious that $\widetilde{d_X}(x,y)=0$ if and only if $x=y$.
	It is also obvious that $\widetilde{d_X}(x,y)=\widetilde{d_X}(y,x)$ for $x,y \in X$.
	We demonstrate the triangle inequality $\widetilde{d_X}(x,y)+\widetilde{d_X}(y,z) \geq \widetilde{d_X}(x,z)$ for $x,y,z \in X$.
	There exists $g \in G$ with $\widetilde{d_X}(x,z)=d_X(gx,gz)$ by Claim 1.
	We have $d_X(x,y) \geq d_X(gx,gy)$ and $\widetilde{d_X}(y,z) \geq d_X(gy,gz)$ by the definition of $\widetilde{d_X}$.
	We get  $\widetilde{d_X}(x,z)=d_X(gx,gz) \leq d_X(gx,gy)+d_X(gy,gz) \leq \widetilde{d_X}(x,y)+\widetilde{d_X}(y,z)$.
	We have demonstrated that the $G$-invariant definable function $\widetilde{d_X}:X \times X \to M_{\geq 0}$ is a continuous distance function.

	\medskip
	Let $\pi:X \to Q$ be the canonical projection.
	We define the definable function $d_Q:Q \times Q \to M$ by 
	\begin{equation}
	d_Q(\overline{x},\overline{y})=\inf\{\widetilde{d_X}(x,y)\;|\;y \in \pi^{-1}(\overline{y})\}, \label{eq:bbb}
	\end{equation}
	where $x \in X$ with $\overline{x}=\pi(x)$.
	Note that $\pi^{-1}(\overline{y})$ is definably compact for any $\overline{y} \in Q$ because $G$ is definably compact.
	As in the proof of  Claim 1, we get the following claim because $\widetilde{d_X}$ is continuous.
	We omit the details of its proof.
	\medskip
	
	\textbf{Claim 2.}
	For any $x \in \pi^{-1}(\overline{x})$, there exists $y \in \pi^{-1}(\overline{y})$ such that $d_Q(\overline{x},\overline{y})=\widetilde{d_X}(x,y)$.
	\medskip
	
		Since $\widetilde{d_X}$ is $G$-invariant, the right hand of (\ref{eq:bbb}) is independent from the choice of $x$ and the function $d_Q$ is well-defined.
	
	We first demonstrate that the definable function $d_Q$ is a definable distance function.
	Let $\overline{x}, \overline{y}, \overline{z} \in Q$ be arbitrary elements.
	It is obvious that $d_Q(\overline{x},\overline{x})=0$.
	We first show that $\overline{x}=\overline{y}$ if $d_Q(\overline{x},\overline{y})=0$.
	There exists $x \in X$ and $y \in X$ with $d_Q(\overline{x},\overline{y})=\widetilde{d_X}(x,y)$ by Claim 2.
	Since $\widetilde{d_X}$ is a distance function, we have $x=y$.
	It implies that $\overline{x}=\overline{y}$.
	
	We next demonstrate the equality $d_Q(\overline{x},\overline{y})=d_Q(\overline{y},\overline{x})$.
	We fix an element $y_0 \in X$ with $\pi(y_0)=\overline{y}$.
	Using the $G$-invariance of $\widetilde{d_X}$, we have 
	\begin{align*}
		d_Q(\overline{x},\overline{y})&=\inf\{\widetilde{d_X}(x,y)\;|\;y \in \pi^{-1}(\overline{y})\}=\inf\{\widetilde{d_X}(x,gy_0)\;|\;g \in G\}\\
		&=\inf\{\widetilde{d_X}(g^{-1}x,y_0)\;|\;g \in G\} = \inf\{\widetilde{d_X}(y_0, gx)\;|\;g \in G\}\\
		&= d_Q(\overline{y},\overline{x}).
	\end{align*}
	We show the triangle inequality $d_Q(\overline{x},\overline{y})+ d_Q(\overline{y},\overline{z}) \geq d_Q(\overline{x},\overline{z})$.
	We fix $x \in \pi^{-1}(\overline{x})$.
	We can take $y \in \pi^{-1}(\overline{y})$ with $d_Q(\overline{x},\overline{y})=\widetilde{d_X}(x,y)$ by Claim 2.
	In the same manner, there exists $z \in \pi^{-1}(\overline{z})$ with $d_Q(\overline{y},\overline{z})=\widetilde{d_X}(y,z)$.
	Since $\widetilde{d_X}$ is a distance function, we obtain
	\begin{align*}
		d_Q(\overline{x},\overline{y})+ d_Q(\overline{y},\overline{z}) &=\widetilde{d_X}(x,y)+\widetilde{d_X}(y,z) \geq \widetilde{d_X}(x,z)\\& \geq \inf\{\widetilde{d_X}(x,z')\;|\; z' \in \pi^{-1}(\overline{z})\}
		=d_Q(\overline{x},\overline{z}).
	\end{align*}
	
	The remaining task is to demonstrate that the distance function $d_Q$ induces the quotient topology $\tau_Q$.
	For that purpose, we have only to show that a definable subset $U$ of $Q$ is open in the definable metric space $(Q,d_Q)$ if and only if $\pi^{-1}(U)$ is open in the definable metric space $(X,d_X)$ by the definition of quotient topology.
	We first demonstrate that $U$ is open assuming that $\pi^{-1}(U)$ is open.
	Take an arbitrary point $\overline{x} \in U$ and a point $x \in \pi^{-1}(\overline{x})$.
	Since $\pi^{-1}(U)$ is open, there exists $\varepsilon >0$ such that $B_X(x,\varepsilon):=\{y \in X\;|\; d_X(x,y)<\varepsilon\} \subseteq \pi^{-1}(U)$.
	We show that $B_Q(\overline{x},\varepsilon):=\{\overline{y} \in Q\;|\;d_Q(\overline{x},\overline{y})<\varepsilon\} \subseteq U$.
	Take an arbitrary point $\overline{y} \in B_Q(\overline{x},\varepsilon)$.
	There exists $y \in \pi^{-1}(\overline{y})$ with $\widetilde{d_X}(x,y) < \varepsilon$ by Claim 2.
	We have $d_X(x,y) \leq \widetilde{d_X}(x,y)$ by the definition of the function $\widetilde{d_X}$.
	They imply that $ y \in B_X(x,\varepsilon) \subseteq \pi^{-1}(U)$.
	We have $\overline{y} =\pi(y) \in \pi(\pi^{-1}(U)) \subseteq U$.
	
	We next prove the opposite implication.
	 Assume that $U$ is open.
	 Take an arbitrary point $x \in \pi^{-1}(U)$ and set $\overline{x}=\pi(x)$.
	 There exists $\varepsilon >0$ such that $B_Q(\overline{x},\varepsilon) \subseteq U$ because $U$ is open in the definable metric space $(Q,d_Q)$.
	 We show the inclusion $\widetilde{B_X}(x,\varepsilon):=\{y \in X\;|\; \widetilde{d_X}(x,y)<\varepsilon\} \subseteq \pi^{-1}(U)$.
	 Let $y \in \widetilde{B_X}(x,\varepsilon)$.
	 We have $d_Q(\overline{x},\pi(y))=\inf\{\widetilde{d_X}(x,y')\;|\; y' \in \pi^{-1}(\pi(y))\} \leq \widetilde{d_X}(x,y)<\varepsilon$.
	 It implies that $\pi(y) \in B_Q(\overline{x},\varepsilon) \subseteq U$.
	 We have $y \in \pi^{-1}(U)$ and we have demonstrated the inclusion $\widetilde{B_X}(x,\varepsilon) \subseteq \pi^{-1}(U)$. 
	 Since $\widetilde{d_X}$ is continuous, the definable set $\widetilde{B_X}(x,\varepsilon) $ is open.
	 Therefore, $\pi^{-1}(U)$ is also open.
\end{proof}

\begin{corollary}
	Let  $\mathcal M=(M,<,+,0,\ldots)$, $(G,d_G)$ and $(X,d_X)$ be as in Theorem \ref{thm:inv_dist}.
	Consider a definable continuous left $G$-action on $X$.
	For any $G$-invariant definable closed subset $A$ of $X$, there exists a definable continuous $G$-invariant function $f:X \to M$ whose zero set is $A$.
\end{corollary}
\begin{proof}
	The definable quotient space $(Q,\tau_Q)$ of $X$ by $G$ is a definable metric space by Theorem \ref{thm:inv_dist}. 
	Let $\pi:X \to Q$ be the canonical projection and $d_Q$ be the definable distance function on $Q$ which induces the topology $\tau_Q$.
	The image $\pi(A)$ is closed by the definition of the quotient topology and the equality $A=\pi^{-1}(\pi(A))$ holds true because $A$ is $G$-invariant.
	Let $\widetilde{f}:Q \to M$ be the definable continuous function given by $\widetilde{f}(x):=\inf\{d_Q(x,a)\;|\; a \in \pi(A)\}$.
	It is obvious that $\widetilde{f}^{-1}(0)=\pi(A)$. 
	The composition $f:= \widetilde{f} \circ \pi$ satisfies the requirement of the corollary.
\end{proof}

\textbf{Acknowledgment.}
The author first employed the assumption that the structure is definably complete locally o-minimal in Lemma \ref{lem:prepre}.
However, the weaker assumption that the structure is definably complete is sufficient.
An anonymous referee informed the author of it.

\end{document}